\documentclass[12pt]{article}
\usepackage{amsfonts,amssymb,amsthm,amsmath,cite,bbm}
\usepackage{mathptmx}
\usepackage{newtxtext}
\usepackage{enumerate}
\usepackage{enumitem}
\usepackage{hyperref}
\usepackage{comment}
\usepackage{tgpagella}

\usepackage{geometry}
\hypersetup{
	colorlinks,%
	citecolor=black,%
	filecolor=black,%
	linkcolor=black,%
	urlcolor=black
}

\newtheorem{theorem}{Theorem}[section]
\newtheorem{lemma}[theorem]{Lemma}

\newtheorem{corollary}[theorem]{Corollary}
\theoremstyle{remark}
\newtheorem{remark}[theorem]{Remark}
\numberwithin{equation}{section}

\newcommand{\N}{{\mathbb N}} 
\newcommand{\R}{{\mathbb R}}
\newcommand{\Rn}{{\mathbb R}^n}
\newcommand{\s}{\mathbb{S}}
\newcommand{\sn}{{\mathbb{S}^{n-1}}}

\newcommand{\K}{{\mathcal K}}
\newcommand{\Kn}{{\mathcal K}^n}

\newcommand{\cFn}{\mathcal{F}^n}
\newcommand{\cEu}{\mathcal{E}_u}
\newcommand{\cMu}{\mathcal{M}_u}
\newcommand{\Sdisk}{\bar{S}}

\newcommand{\hm}{\mathcal H}

\newcommand{\fconv}{{\mbox{\rm Conv}(\R^n)}} 
\newcommand{\fconvs}{{\mbox{\rm Conv}_{{\rm sc}}(\R^n)}} 
\newcommand{\fconvf}{{\mbox{\rm Conv}(\R^n; \R)}}
\newcommand{\fconvcd}{{\mbox{\rm Conv}_{{\rm cd}}(\R^n)}} 

\newcommand{\proj}{\operatorname{proj}}
\newcommand{\gnom}{\operatorname{gno}} 

\newcommand{\infconv}{\mathbin{\Box}} 
\newcommand{\sq}{\mathbin{\vcenter{\hbox{\rule{.3ex}{.3ex}}}}} 

\DeclareMathOperator*{\argmin}{arg\,min}

\newcommand{\MA}{\text{\rm MA}} 
\newcommand{\MAp}{\text{\rm MA}^{\!*}} 

\newcommand{\oZZb}[2]{\overline{\operatorname{V}}_{#1,#2}} 

\newcommand{\Gr}{\operatorname{G}}
\newcommand{\Grass}[2]{\operatorname{G}(#2,#1)}
\newcommand{\Aff}[2]{\operatorname{A}(#2,#1)}
\newcommand{\SO}{\operatorname{SO}}

\renewcommand{\d}{\,\mathrm{d}}
\newcommand{\Hess}{{\operatorname{D}}^2}
\newcommand{\ind}{{\rm\bf I}}
\newcommand{\interior}{\operatorname{int}} 
\newcommand{\dom}{\operatorname{dom}} 
\newcommand{\epi}{\operatorname{epi}} 
\newcommand{\supp}{\operatorname{supp}} 
\newcommand{\cl}{\operatorname{cl}} 
\newcommand{\ext}{\operatorname{ext}} 
\newcommand{\pos}{\operatorname{pos}} 
\newcommand{\lin}{\operatorname{span}} 
\newcommand{\reg}{\operatorname{reg}} 

\title{Kubota-type Formulas and Supports\\of Mixed Measures}
\author{Daniel Hug, Fabian Mussnig, and Jacopo Ulivelli}

\date{}

\begin{document}
\maketitle

\begin{abstract}
Kubota’s integral formula expresses the intrinsic volumes of a convex body as averages over its projections onto linear subspaces. In this work, we introduce a new class of Kubota-type formulas for mixed area measures adapted to rotations around a fixed axis, which encode a crucial disintegration property. Our construction is motivated by applications to valuations on convex functions. In the latter framework, we obtain corresponding statements for (conjugate) mixed Monge--Amp\`ere measures. As a by-product, we characterize supports of mixed area and mixed Monge--Amp\`ere measures, thereby confirming a special case of a conjecture by Schneider.

\bigskip

{\noindent
}
\end{abstract}

\section{Introduction and Main Results}

The interplay between geometry and analysis is a central theme in modern convexity (see, for example \cite[Chapter 9]{Artstein_Giannopoulos_Milman_II} or \cite{Gardner_BM}). In the geometric setting, we first develop new Kubota-type formulas for mixed area measures, which are particularly suited for situations that involve rotations around a fixed axis. We use these identities to characterize the supports of a special family of mixed area measures, which allows us to confirm a long-standing conjecture by Schneider in a special case. After this, we develop the necessary tools to transfer our geometric results to the analytic setting. This allows us to derive analogous integral geometric identities and characterizations of supports of (conjugate) mixed Monge--Amp\`ere measures of convex functions, highlighting their natural behavior under rotations and projections/restrictions. These results not only reveal a common underlying structure between geometry and analysis but also have direct implications for the theory of valuations on convex functions.

\label{se:introduction}
\paragraph*{Mixed Area Measures}
Let $\Kn$, $n \in \N$, denote the set of non-empty, compact, convex subsets of $\Rn$, whose elements are commonly referred to as \emph{convex bodies}. For a convex body $K\in\Kn$, we denote by $S_{n-1}(K,\cdot)$ its \emph{surface area measure}. If $K$ has non-empty interior and $\omega$ is a Borel subset of the unit sphere $\sn$, then $S_{n-1}(K,\omega)$ is simply the ${(n-1)}$-dimensional Hausdorff measure of the set of all boundary points $x\in \partial K$ at which $K$ has an outer unit normal in $\omega$.
The surface area measure plays a fundamental role in convexity, most prominently in the Minkowski problem, where one asks under which conditions a given Borel measure on $\sn$ is the surface area measure of a convex body $K$. In particular, every full-dimensional body $K$ is uniquely determined by $S_{n-1}(K,\cdot)$ up to translations. See, for example, \cite{Boeroeczky_LYZ_log_Mink,Boeroeczky_LYZ_Zhao_Gauss,Gardner_Hug_Weil_Xing_Ye_2019,Huang_Liu_Xi_Zhao_dual_LC,Huang_LYZ_Acta,Ivaki_Milman_Uniqueness,Kryvonos_Langharst,Livshyts_Minkowski,Mui_Aleksandrov, Haberl_LYZ} for some recent developments and further references.

\medskip

More generally, we consider the \emph{mixed area measure} $S$, which is the unique symmetric map from $(\Kn)^{n-1}$ into the space of finite Borel measures on $\sn$ such that
\begin{equation}
\label{eq:def_mixed_s}
S_{n-1}(\lambda_1 K_1 + \cdots + \lambda_m K_m, \cdot) = \sum_{i_1,\ldots, i_{n-1}=1}^m \lambda_{i_1}\cdots \lambda_{i_{n-1}} S(K_{i_1},\ldots, K_{i_{n-1}},\cdot)
\end{equation}
for every $K_1,\ldots,K_m\in\Kn$, $\lambda_1\ldots, \lambda_m\geq 0$, and $m\in \N$. Here,
$$\lambda_1 K_1 + \cdots + \lambda_m K_m=\{ \lambda_1x_1+\cdots+\lambda_mx_m: x_i\in K_i \, \text{ for }i=1,\ldots,m\}$$
is a \emph{Minkowski combination} of the given convex bodies $K_1,\ldots,K_m$ (the \emph{Minkowski sum} of the convex bodies $\lambda_1K_1,\ldots,\lambda_mK_m$). A particularly relevant family of mixed measures are the $j$th \emph{area measures} $S_j$, $j\in\{0,\ldots,n-1\}$, which are given by
\begin{equation}
\label{eq:s_j_mixed_s}
S_j(K,\cdot)=S(K[j],B^n[n-1-j],\cdot),
\end{equation}
where on the right-hand side the convex body $K$ is repeated $j$ times and the Euclidean unit ball $B^n$ is repeated $(n-1-j)$ times. When $1<j<n-1$, the associated Minkowski-type problem for these measures is known as the Christoffel--Minkowski problem, a central open question in convex geometry and geometric analysis that is still open in full generality. See, for example, \cite{guan_ma_chris_mink_I,hyz_bams}.

\paragraph*{Kubota-type Formulas}

Let $\{e_1,\ldots,e_n\}$ denote the standard orthonormal basis of $\Rn$. In the following we write $\ell$ for $\lin\{e_n\}$, i.e., the ``vertical line'', and set $L=\ell^\perp\cong\R^{n-1}$. The geometric results of this article concern the measures that can be understood as variants of \eqref{eq:s_j_mixed_s}, where the unit ball $B^n$ is replaced by
$$B_L^{n-1}=\{(x_1,\ldots,x_n)\in\Rn : x_1^2 + \cdots + x_{n-1}^2 \leq 1, x_n=0\},$$
which in other words is the $(n-1)$-dimensional unit ball in $L$. For this purpose, we define
\[
\Sdisk_j(K,\cdot)=S(K[j], B_L^{n-1}[n-1-j],\cdot)
\]
for $K\in\Kn$ and $j\in\{0,\ldots,n-1\}$. While our primary interest in these measures stems from their applications in the analytic setting (see below), let us briefly emphasize their significance. As recently shown by Brauner, Hofst\"atter, and Ortega-Moreno in\cite{Brauner_Hofstaetter_Ortega-Moreno}, whose preprint appeared after a first version of this article was posted on arXiv, the measures $\Sdisk_j(K,\cdot)$ occur naturally in the investigation of zonal valuations on convex bodies (cf.\ \cite{Knoerr_zonal,Schuster_Wannerer}), where, in contrast to the usual area measures, they allow a representation without singular densities. Shortly afterwards, the same authors presented a solution to the Christoffel--Minkowski problem for bodies of revolution in \cite{Brauner_Hofstaetter_Ortega-Moreno_CM}, which is based on these measures and in which the following Kubota-type formula, the first main result of this article, plays a central role. For further recent applications of this formula to Minkowski problems, see also \cite{Brauner_Hofstaetter_Ortega-Moreno_Cdisk,Mussnig_Ulivelli_CM}.

\begin{theorem}
\label{thm:ck_sam}
If $1\leq j \leq n-1$ and $f\colon \sn\to[0,\infty)$ is measurable, then
\begin{multline*}
\frac{1}{\kappa_{n-1}} \int_{\sn} f(z) \d S(K_1,\ldots,K_j,B_L^{n-1}[n-1-j],z)\\
=\frac{1}{\kappa_j} \int_{\Gr(\ell,j+1) } \int_{\s_E^{j}} f(z) \d S_E(\proj_E K_1,\ldots,\proj_E K_j,z)\d E
\end{multline*}
for $K_1,\ldots,K_j\in\Kn$.
\end{theorem}
\noindent
Here, $\Gr(\ell,j+1)$ denotes the Grassmannian of all $(j+1)$-dimensional linear subspaces of $\Rn$ that contain the line $\ell$ and integration on this set is understood with respect to the unique probability measure that is invariant under those rotations that map $\ell$ onto itself. In addition, we write $\proj_E K$ for the orthogonal projection of $K\in\Kn$ onto $E\in\Gr(\ell,j+1)$. To prevent misconceptions, we further denote the mixed area measure with respect to the ambient space $E$ by $S_E$ and remark that it is a Borel measure on $\s_E^{j}$, i.e., the $j$-dimensional unit sphere in $E$. Lastly, the constant $\kappa_i=\pi^{i/2}/\Gamma((i/2)+1)$ denotes the $i$-dimensional volume of the unit ball in $\R^i$ for $i\in\{0,\ldots,n\}$.

Let us point out that for $K_1=\cdots=K_j=K$ we can rewrite Theorem~\ref{thm:ck_sam} as
\begin{equation}
\label{eq:bar_sk_rewrite}
\Sdisk_j(K,\omega) = \frac{\kappa_{n-1}}{\kappa_j} \int_{\SO(n-1)} S_j'(\proj_{\vartheta \bar{E}_{j+1}} K,\omega \cap \vartheta \bar{E}_{j+1}) \d\vartheta
\end{equation}
for Borel sets $\omega \subseteq \sn$, where $\SO(n-1)$ is the group of rotations that fix $e_n$ and integration is understood with respect to the Haar probability measure on this group. Furthermore, $\bar{E}_{j+1}$ is some fixed element of $\Gr(\ell,j+1)$ and $S_j'$ denotes the usual surface area measure in the ambient space $\vartheta \bar{E}_{j+1}$. Equation \eqref{eq:bar_sk_rewrite}, which should be compared with \cite[Equation (5.4)]{Schneider_75}, not only shows that the measures $\Sdisk_j(K,\cdot)$ naturally encode the action of $\SO(n-1)$, but also provides a crucial disintegration property for these measures. They are, therefore, ideal tools to study bodies of revolution.

We remark that the existence of formulas of the above type was already indicated in \cite[Section 32]{Bonnesen_Fenchel}. Moreover, a new proof of Theorem~\ref{thm:ck_sam} was obtained in the aforementioned article \cite{Brauner_Hofstaetter_Ortega-Moreno} as a consequence of a Hadwiger-type theorem together with further integral geometric identities for the measures $\bar{S}_j(K,\cdot)$. Among these results is an additive kinematic formula for these measures, which in a first version was previously shown in \cite{Hug_Mussnig_Ulivelli_2} by the authors of the present article. For another recent integral geometric formula, motivated by problems in stereology and employing averaging over subspaces that contain a fixed subspace, see \cite[Theorem 2]{Dare_Kiderlen}. 

\paragraph*{Supports} 
Of particular interest is the support of a (mixed) area measure, which is defined as the complement of the largest open subset of the unit sphere on which the measure vanishes. The support of $S_{n-1}(K,\cdot)$ is geometrically characterized as the closure of the set of \emph{extreme unit normal vectors} of $K$. These are the vectors $z\in\sn$ which are normal vectors to $K$ at some $x\in\partial K$ and such that $z$ cannot be written as the sum of two linearly independent normal vectors of $K$ at $x$. We refer to Section~\ref{se:desc_supp} for details.

The situation becomes considerably more complex for general mixed area measures. Finding a geometric description of the support of $S(K_1,\ldots,K_{n-1},\cdot)$ with arbitrary $K_1,\ldots,K_{n-1}\in\Kn$ is a long-standing open problem, with highly relevant applications. Most notably, it characterizes the equality cases in the monotonicity of mixed volumes and appears to be a crucial point in understanding the equality cases of the Alexandrov--Fenchel inequality \cite{Bernig_Kotrbaty_Wannerer,Hug_Reichert_AF,Kotrbaty_Wannerer_HR_AF,shenfeld_van_handel_acta}. In 1985, Schneider~\cite{Schneider_85} conjectured that the support of $S(K_1,\ldots,K_{n-1},\cdot)$ is the closure of the set of \emph{\((K_1,\ldots,K_{n-1})\)-extreme vectors}, which are denoted by $\ext(K_1,\ldots,K_{n-1})$ and which we define in Section~\ref{se:proof_thm_supp_ext}.

\medskip

So far, Schneider's conjecture has been confirmed in a few cases. Shenfeld and van Handel~\cite{shenfeld_van_handel_acta} treated tuples of the form $(B^n,C_1,\ldots,C_{n-2})$, where the bodies $C_i$ are either zonoids or smooth. More recently, the first-named author and Reichert generalized this to the larger class of polyoids \cite{Hug_Reichert_support}. Very recently, and after a first version of this article appeared, van Handel and Wang \cite{van_handel_wang} proved that Schneider's conjecture holds for tuples of the form $(K,L,\ldots,L)$ for arbitrary convex bodies $K,L\in\Kn$, thereby fully settling the three-dimensional case.

Long before these recent advances and predating his general conjecture by ten years, Schneider \cite{Schneider_75} described the support of the area measure $S_j(K,\cdot)$ as the closure of the set of all \emph{$(n-1-j)$-extreme unit normal vectors} of $K$ (see Theorem~\ref{thm:supp_s_j} below). He thereby confirmed a conjecture of Weil \cite{weil_73}, who treated the special case $j=1$. This implies the nesting
\begin{equation}
\label{eq:supp_s_j_nested}
\supp S_k(K,\cdot)\subseteq \supp S_j(K,\cdot)
\end{equation}
for $0\leq j\leq k\leq n-1$. For $j=0$ this relation is trivial, since $S_0(K,\cdot)=S_{n-1}(B^n,\cdot)$ and hence has full support. By contrast, the support of $S_{n-1}(K,\cdot)$ might be very small. For example, when $K$ is a polytope it is supported on the facet normals. Firey also raised the validity of this inclusion in his invited address at the 1974 ICM in Vancouver \cite{Firey_ICM}.

\medskip

We confirm the following case of Schneider's conjecture, where $\cl(A)$ denotes the closure of $A\subseteq \sn$. Recall that $\Sdisk_j(K,\cdot)=S(K[j], B_L^{n-1}[n-1-j],\cdot)$ for $K\in\Kn$ and $1\leq j\leq n-1$.

\begin{theorem}
\label{thm:supp_ext}
If $1\leq j\leq n-1$, then
$$\supp \Sdisk_j(K,\cdot)=\cl \big(\ext(K[j],B_L^{n-1}[n-1-j])\big)$$
for $K\in\Kn$.
\end{theorem}
The inclusion
\begin{equation}
\label{eq:basicsupportinclusion}
\supp \Sdisk_j(K,\cdot)\subseteq \cl\big(\ext(K[j],B_L^{n-1}[n-1-j])\big)
\end{equation}
has already been established by Schneider \cite[Proposition 3.8]{Schneider_88}, where it appears as a special case. Far more generally, van Handel and Wang showed in the above-mentioned \cite{van_handel_wang} that 
\[
\supp S(K_1,\ldots,K_{n-1}) \subseteq \cl\big(\ext(K_1,\ldots,K_{n-1})\big)
\]
holds for any $K_1,\ldots,K_{n-1}\in\Kn$. The proof that is presented here is independent of \cite{Schneider_88} and uses the new Kubota-type formula Theorem~\ref{thm:ck_sam}, as well as a description of $\supp \Sdisk_j(K,\cdot)$ in terms of extreme unit normal vectors of projections of $K$ onto subspaces of dimension $j+1$ (see Theorem~\ref{thm:descrp_supp_s_proj} below). The strategy of our proof is somewhat similar to Schneider's treatment of the classical area measures in \cite{Schneider_75}.

As a consequence of Theorem~\ref{thm:supp_ext}, we obtain the following nesting of supports, which is analogous to \eqref{eq:supp_s_j_nested}.

\begin{corollary}
\label{cor:supp_s_nested}
If $1\leq j \leq k \leq n-1$, then
$$\supp \Sdisk_k(K,\cdot)\subseteq \supp \Sdisk_j(K,\cdot)$$
for $K\in\Kn$.
\end{corollary}
Note that the case $j=0$ is excluded from Corollary~\ref{cor:supp_s_nested} as the statement fails in that case. Indeed, $\supp \Sdisk_0(K,\cdot)= \supp S(B_L^{n-1}[n-1],\cdot)=\{\pm e_n\}$, while the support of $\Sdisk_k(K,\cdot)$ might be all of $\sn$. For example, when $k=n-1$ and $K=B^n$.

\paragraph*{Mixed Monge--Amp\`ere Measures}
The geometric results above are strongly motivated by their applications to the analytic setting. For a convex function $v\colon\Rn\to\R$ we denote by $\MA(v;\cdot)$ the \emph{Monge--Amp\`ere measure} associated with $v$, which is a Radon measure on $\Rn$ (see Section~\ref{se:convex_functions} for details). If $v$ is additionally assumed to be in $C^2(\Rn)$, then $\MA(v;\cdot)$ is absolutely continuous with respect to the Lebesgue measure on $\Rn$ and
\begin{equation}
\label{eq:MA_det_hess}
\d\MA(v;x)=\det(\Hess v(x))\d x,
\end{equation}
where $\Hess v(x)$ denotes the Hessian matrix of $v$ at $x\in\Rn$. Similar to the definition of the mixed area measure, the \emph{mixed Monge--Amp\`ere measure} $\MA(w_1,\ldots,w_n;\cdot)$ of convex functions $w_1,\ldots,w_n\colon \Rn\to\R$ is defined by the relation
\begin{equation}
\label{eq:def_mixed_ma}
\MA(\lambda_1 v_1+\cdots + \lambda_m v_m;\cdot) = \sum_{i_1,\ldots,i_n=1}^m \lambda_{i_1}\cdots \lambda_{i_n} \MA(v_{i_1},\ldots,v_{i_n};\cdot),
\end{equation}
where $m\in\N$, $v_1,\ldots,v_m\colon \Rn\to\R$ are convex, and $\lambda_1,\ldots,\lambda_m\geq 0$.
Here, we assume in addition that this measure is symmetric in its entries.

\medskip

In Section~\ref{se:properties_ma}, we establish new connections of mixed Monge--Amp\`ere measures with area measures and mixed volumes of convex bodies in both $n$ and $(n+1)$-dimensional space and study their behavior when we restrict functions to linear subspaces. We then use these insights together with the Kubota-type formula Theorem~\ref{thm:ck_sam} to obtain a corresponding statement for mixed Monge--Amp\`ere measures. The latter directly improves the main results of \cite{Colesanti-Ludwig-Mussnig-6}, where such formulas were established under the additional assumption that the integrand is rotationally invariant.

Let
$$\fconvf=\{v\colon \Rn\to\R : v \text{ is convex}\}$$
denote the convex cone of real-valued convex functions on $\Rn$. 
\begin{theorem}
\label{thm:ck_ma}
If $1\leq j< n$ and $\varphi\colon\Rn\to[0,\infty)$ is measurable, then
\begin{multline*}
\frac{1}{\kappa_n} \int_{\Rn} \varphi(x)\d\MA(v_1,\ldots,v_j,h_{B^n}[n-j];x)\\
=\frac{1}{\kappa_j}\int_{\Grass{j}{n}}\int_E \varphi(x_E) \d\MA_E(v_1\vert_E,\ldots, v_j\vert_E;x_E)\d E
\end{multline*}
for $v_1,\ldots,v_j\in\fconvf$.
\end{theorem}
Here, $h_{B^n}$ is the support function of $B^n$, that is, $h_{B^n}(x)=|x|$ for $x\in\Rn$, where $|\cdot|$ denotes the Euclidean norm (see Section~\ref{se:convex_bodies} for details). Furthermore, $\Grass{j}{n}$ denotes the Grassmannian of $j$-di\-men\-sion\-al linear subspaces of $\Rn$ and integration on $\Grass{j}{n}$ is always understood with respect to the Haar probability measure on this space. In addition, for given $E\in\Grass{j}{n}$ and $v_1,\ldots,v_j\in\fconvf$, we write $\MA_E(v_1\vert_E,\ldots,v_j\vert_E;\cdot)$ for the mixed Monge--Amp\`ere measure with respect to the ambient $j$-dimensional space $E$, of the restrictions of $v_1,\ldots,v_j$ to $E$.

\medskip

Of particular interest is the family of measures $\MA(v[j],h_{B^n}[n-j],\cdot)$ with $j\in\{0,\ldots,n\}$ and $v\in\fconvf$, which were studied in more detail in \cite{Colesanti-Ludwig-Mussnig-7}. Much like Theorem~\ref{thm:ck_sam} and equation \eqref{eq:bar_sk_rewrite}, Theorem~\ref{thm:ck_ma} offers a deeper insight by showing that these measures arise naturally from the usual Monge--Amp\`ere measure (in arbitrary dimension) together with the action of the rotation group. This fundamental structural result was subsequently applied to convex solutions of $k$-Hessian equations (cf. \cite{Trudinger_Wang_Hessian_I,Trudinger_Wang_Hessian_II}) in \cite{Mussnig_Ulivelli_CM}. Theorem~\ref{thm:ck_ma} as well as results from Section~\ref{se:properties_ma} were furthermore already utilized in \cite{Hug_Mussnig_Ulivelli_2,Mussnig_Ulivelli,Mouamine_Mussnig_2}, shortly after the first preprint of this work appeared on arXiv.

Concerning the supports of the aformentioned Monge--Amp\`ere-type measures, we obtain the following result as a consequence of Corollary~\ref{cor:supp_s_nested}.
\begin{theorem}
\label{thm:ma_supp}
If $1\leq j \leq k \leq n$, then
$$\supp \MA(v[k],h_{B^n}[n-k];\cdot)\subseteq \supp \MA(v[j],h_{B^n}[n-j];\cdot)$$
for $v\in\fconvf$.
\end{theorem}
In addition to the above, we give a description of $\supp \MA(v[j],h_{B^n}[n-j];\cdot)$ in terms of restrictions of $v$ to $j$-dimensional subspaces in Theorem~\ref{thm:supp_MAj_rest} below.

\paragraph*{Functional Intrinsic Volumes}
Our interest in mixed Monge--Amp\`ere measures stems from the fact that they play a crucial role in the emergent theory of valuations on convex functions \cite{Alesker_19,Colesanti-Ludwig-Mussnig-2,Colesanti-Ludwig-Mussnig-3,Colesanti-Ludwig-Mussnig-4,Colesanti-Ludwig-Mussnig-5,Colesanti-Ludwig-Mussnig-7,Colesanti-Ludwig-Mussnig-8,Colesanti-Ludwig-Mussnig-6,Hofstaetter_Knoerr,Hofstaetter_Schuster_BS,Knoerr_support,Knoerr_singular,Knoerr_smooth,Knoerr_Ulivelli,Li_Legendre,Milman_AGA_perspectives,Mouamine_Mussnig_2,Mussnig_polar}. It is one of the most active topics within the current trend of finding functional analogs of classical concepts and results from geometry \cite{Hoehner_symm,Hug_Mussnig_Ulivelli_2,Li_Mussnig,Li_Schuett_Werner_floating,Li_SLn,Milman_Rotem_mixed,Mussnig_Ulivelli,Rotem_Riesz,Roysdon_Xing}. We use Theorem~\ref{thm:ma_supp} to establish an elementary property for (renormalized) functional intrinsic volumes, which we define as follows. For $0\leq j\leq n$ and $\alpha\in C_c({[0,\infty)})$, where $C_c({[0,\infty)})$ denotes the set of continuous functions on $[0,\infty)$ with compact support, let $\oZZb{j}{\alpha}^*\colon \fconvf\to\R$ be the $j$th \emph{functional intrinsic volume} with density $\alpha$, which is given by
\begin{equation}
\label{eq:ozzb_ma}
\oZZb{j}{\alpha}^*(v)=\binom{n}{j}\frac{1}{\kappa_{n-j}}\int_{\Rn} \alpha(|x|)\d \MA(v[j],h_{B^n}[n-j];x)
\end{equation}
for $v\in\fconvf$. Note that $\MA(h_{B^n};\cdot) = \kappa_n \delta_{o}$,
where $\delta_o$ denotes the Dirac measure at the origin. Thus, $\oZZb{0}{\alpha}^*$ is constant and $\oZZb{0}{\alpha}^*(v)=\alpha(0)$ for every $v\in\fconvf$.

The operators $\oZZb{j}{\alpha}^*$ share many properties with the classical intrinsic volumes $V_j$ and can be seen as their generalizations to the functional setting. Most strikingly, they were characterized by a Hadwiger-type theorem in \cite{Colesanti-Ludwig-Mussnig-5}. See also \cite{Colesanti-Ludwig-Mussnig-8,Knoerr_singular} for alternative proofs. We remark that the original definition of functional intrinsic volumes uses Hessian measures, and the equivalence with definition \eqref{eq:ozzb_ma} is a consequence of \cite[Theorem 2.5]{Colesanti-Ludwig-Mussnig-7}.

\medskip

It is an elementary property of intrinsic volumes that if $K\in\Kn$ is such that $V_j(K)=0$, then also $V_k(K)=0$ for every $1\leq j\leq k\leq n$.
We obtain the following analog of this as a direct consequence of Theorem~\ref{thm:ma_supp} and \eqref{eq:ozzb_ma}, and refer to \cite{Mussnig_Ulivelli} for an application to inequalities.

\begin{corollary}
\label{cor:ma_supp_nested}
Let $1\leq j\leq k\leq n$ and let $\alpha\in C_c({[0,\infty)})$ be non-negative. If $v\in\fconvf$ is such that $\oZZb{j}{\alpha}^*(v)=0$, then also $\oZZb{k}{\alpha}^*(v)=0$.
\end{corollary}

\paragraph{Overview of the Paper} In Section~\ref{se:preliminaries} we explain our notation and list preliminary results. Section~\ref{se:supports} is dedicated to the main results on mixed area measures. First, we establish the Kubota-type formula Theorem~\ref{thm:ck_sam} in Section~\ref{se:kubota}. We will then use this result in Section~\ref{se:desc_supp} to prove Theorem~\ref{thm:descrp_supp_s_proj}, which shows that the supports of the mixed area measures that are considered in this article can be obtained from extreme unit vectors of projections of convex bodies. The proof of Theorem~\ref{thm:supp_ext} and Corollary~\ref{cor:supp_s_nested} is then completed in Section~\ref{se:proof_thm_supp_ext}. In the following Section~\ref{se:properties_ma} we establish several new properties of (conjugate) mixed Monge--Amp\`ere measures: first, we explore connections of these measures with mixed area measures and mixed volumes on $\Kn$ in Section~\ref{se:connections_ma}. In Section~\ref{se:connections_ma2}, however, we will prove fundamental connections with mixed areas measures of $(n+1)$-dimensional convex bodies. A result that explains how conjugate mixed Monge--Amp\`ere measures behave under projections of convex functions is the subject of Section~\ref{se:behavior_proj}. Our main findings for (conjugate) mixed Monge--Amp\`ere measures are then treated in Section~\ref{se:main_res_ma}. An equivalent version of Theorem~\ref{thm:ck_ma} for super-coercive convex functions is established in Section~\ref{se:kubota_ma} and compared with a Crofton-type formula, due to \cite{Colesanti-Hug_MM}, in Section~\ref{se:compare_crofton}. Lastly, the proof of Theorem~\ref{thm:ma_supp} can be found in Section~\ref{se:supp_ma}.

\section{Preliminaries}
\label{se:preliminaries}

We collect some results on convex bodies and convex functions, and we refer to \cite{Gardner_GT,Hug_Weil_Lectures,RockafellarWets,Schneider_CB} as general references for these topics. More specific results are marked with precise references.

\subsection{Background on Convex Bodies}
\label{se:convex_bodies}
For basic notation, recall that the set of convex bodies, $\Kn$, the area measures $S_j$, and the mixed area measure $S$ were already defined in Section~\ref{se:introduction}. The \emph{mixed volume} $V\colon (\Kn)^n\to\R$ is the unique map such that
$$
V_n(\lambda_1 K_1 + \cdots + \lambda_m K_m) = \sum_{i_1,\ldots,i_n=1}^m \lambda_{i_1}\cdots \lambda_{i_n} V(K_{i_1},\ldots,K_{i_n})
$$
for $K_1,\ldots,K_m\in\Kn$, $\lambda_1,\ldots,\lambda_m\geq 0$, and $m\in\N$, where we assume in addition that $V$ is symmetric in its entries. Here, $V_n(K)$ denotes the volume of $K\in\Kn$, i.e., the $n$-dimensional Lebesgue measure, and $V(K,\ldots,K)=V_n(K)$. A special family of mixed volumes is given by the \emph{intrinsic volumes}
\begin{equation}
\label{eq:intrinsic_mixed}
V_j(K) = \binom{n}{j}\frac{1}{\kappa_{n-j}} V(K[j],B^n[n-j])
\end{equation}
with $0\leq j\leq n$.

\medskip

Every convex body $K\in\Kn$ is uniquely determined by its \emph{support function} $h_K(x)=\sup_{y\in K} \langle x,y\rangle$ with $x\in\Rn$, where $\langle \cdot\,,\cdot\rangle$ denotes the usual scalar product of $x,y\in\Rn$. Observe that the support function is linear with respect to (non-negative) Minkowski combinations, that is,
$$h_{\lambda K + \mu L} = \lambda h_K + \mu h_L$$
for $K,L\in\Kn$ and $\lambda,\mu\geq 0$. Mixed volumes and mixed area measures are connected via the well-known relation
$$V(K_1,\ldots,K_n)=\frac 1n \int_{\sn} h_{K_1}(z) \d S(K_2,\ldots,K_n,z)$$
for $K_1,\ldots,K_n\in\Kn$. Since the left side is symmetric in its entries, so is the right side. It is now easy to see that the mixed volume and the mixed area measure are also linear with respect to Minkowski addition in each of their entries. Moreover, for given $K_2,\ldots,K_n\in\Kn$, the mixed area measure $S(K_2,\ldots,K_n,\cdot)$ is uniquely determined by the values of all mixed volumes $V(K_1,\ldots,K_n)$ with $K_1\in \Kn$. 

\medskip

Occasionally, we need to deal with continuous functionals on $\Kn$. The space of convex bodies is equipped with the topology induced by the \emph{Hausdorff metric}. This means that a sequence of bodies $K_i\in\Kn$, $i\in\N$, converges to a body $K\in\Kn$ if and only if
$$\sup\left\{ \left| h_{K_i}(z)-h_K(z) \right|:z\in\sn\right\}\to 0$$
as $i\to\infty$.

\medskip

We write $[o,x]$ for the line segment that connects the origin $o$ with $x\in\Rn$. For the following result, see, for example, \cite[Proposition 2.15]{Hug_Reichert_support}.

\begin{lemma}
\label{le:mx_area_meas_proj}
If $e\in\sn$ and $E=e^\perp$, then
$$(n-1)S(K_1,\ldots,K_{n-2},[o,e],B)=S_E(\proj_E K_1,\ldots,\proj_E K_{n-2}, B\cap E),$$
for $K_1,\ldots,K_{n-2}\in\Kn$ and Borel sets $B\subseteq \sn$. 
In particular, if $c\in\R$, then 
$$S(K_1,\ldots,K_{n-2},[o,c \cdot e_n],\cdot)=0$$
on $\s^{n-1}_- =\{\nu\in\sn : \langle \nu,e_{n}\rangle <0\}$.
\end{lemma}

For the next result, we refer to \cite[Section 5.1]{Schneider_CB}.

\begin{lemma}
\label{le:properties_int_s}
Let $0\leq j \leq n-1$ and $K_1,\ldots,K_{n-1-j}\in\Kn$. If $f\in C(\s^{n-1})$, then
$$K\mapsto \int_{\s^{n-1}} f(z) \d S(K_1,\ldots,K_{n-1-j},K[j],z)$$
defines a continuous map on $\Kn$.
\end{lemma}

We close this subsection with the following result due to Hadwiger \cite{Hadwiger}, where we say that $\psi\colon\Kn\to\R$ is \emph{Minkowski additive} if
$$\psi(K+L)=\psi(K)+\psi(L)$$
for $K,L\in\Kn$. For the version that is presented here, we refer to \cite[Theorem 3.11]{Hug_Weil_Lectures} or  \cite[Theorem 3.3.2]{Schneider_CB}. In the following, we write $\SO(n)$ for the group of proper (orientation preserving) rotations of $\R^n$. 
\begin{lemma}
\label{le:char_mean_width}
Let $n\geq 2$ and $\psi\colon \Kn\to\R$. If $\psi$ is Minkowski additive, invariant under proper rotations, and continuous at $B^n$, then $\psi$ is a constant multiple of $V_1$.
\end{lemma}

\subsection{Background on Convex Functions}
\label{se:convex_functions}
Let $\fconv$ denote the space of all lower semicontinuous (l.s.c.), proper, convex functions $w\colon \Rn\to(-\infty,\infty]$, where we say that $w$ is \emph{proper} if there exists $x\in\Rn$ such that $w(x)<\infty$. For such a function $w$, we write
$$\partial w(x)=\{y\in\Rn : w(z)\geq w(x)+\langle y,z-x\rangle \; \forall z\in\Rn\}$$
for the \emph{subdifferential} of $w$ at $x\in\Rn$. Each of its elements is a \emph{subgradient} of $w$ at $x$, and if $w$ is differentiable at $x$, then $\partial w(x)$ only contains the usual gradient of $w$ at $x$, which we denote by $\nabla w(x)$.

As the next result shows, the subdifferential of the pointwise sum of functions is the Minkowski sum of the respective subdifferentials (see, for example, \cite[Corollary 10.9]{RockafellarWets}). Recall that $\fconvf$ contains those elements of $\fconv$ that only take values in $\R$.

\begin{lemma}
\label{le:sum_subdiff}
If $v_1,\ldots,v_m\in\fconvf$, then
$$\partial(v_1+\cdots+v_m)(x)=\partial v_1(x) + \cdots + \partial v_m(x)$$
for $x\in\Rn$.
\end{lemma}

\medskip

For $v\in\fconvf$, the \emph{Monge--Amp\`ere measure} associated with $v$ is given by
\begin{equation}
\label{eq:def_MA}
\MA(v;B)=V_n\left(\bigcup_{b\in B} \partial v(b) \right)
\end{equation}
for Borel sets $B\subseteq\Rn$. This measure satisfies \eqref{eq:MA_det_hess} (see also \cite{Figalli_MA}) and gives rise to the mixed Monge--Amp\`ere measure, which we defined in \eqref{eq:def_mixed_ma}. It is straightforward to see from \eqref{eq:def_mixed_ma} that for $m\in\N$, $v_1,\ldots,v_m\in\fconvf$, and $\lambda_1,\ldots,\lambda_m\geq 0$ the relation
\begin{multline}
\label{eq:mixed_ma_multinom}
\MA(\lambda_1 v_1+\cdots + \lambda_m v_m;\cdot)\\
= \sum_{r_1,\ldots,r_m=0}^n \binom{n}{r_1,\ldots, r_m} \lambda_1^{r_1} \cdots \lambda_m^{r_m} \MA(v_1[r_1],\ldots,v_m[r_m];\cdot)
\end{multline}
holds. Here, we use the multinomial coefficient
$$\binom{n}{r_1,\ldots, r_m} = \begin{cases}\frac{n!}{r_1! \cdots r_m!} \quad &\text{if } \sum_{j=1}^m r_j=n \text{ and } r_j\in\{0,1,\ldots,n\},\\
0 \quad &\text{else}.\end{cases}$$
Equivalently, we can represent the mixed Monge--Amp\`ere measure of\linebreak$v_1,\ldots,v_n\in\fconvf$ with the polarization formula
\begin{equation}
\label{eq:polarization_formula}
\MA(v_1,\ldots,v_n;\cdot)=\frac{1}{n!}\sum_{k=1}^n \sum_{1\leq i_1 < \cdots < i_k \leq n} (-1)^{n-k} \MA(v_{i_1}+\cdots + v_{i_k};\cdot).
\end{equation}
In each of its entries, this measure is additive and positively homogeneous of degree 1, that is,
\begin{equation}
\label{eq:ma_linear}
\MA(\lambda v + \mu w, v_2,\ldots,v_n;\cdot) = \lambda \MA(v,v_2,\ldots,v_n;\cdot) + \mu \MA(w,v_2,\ldots,v_n;\cdot)
\end{equation}
for $v,w,v_2,\ldots,v_n\in\fconvf$ and $\lambda,\mu\geq 0$. See, for example, \cite[Theorem 4.3 (f)]{Colesanti-Ludwig-Mussnig-7}.

If $q(x)=|x|^2/2$ and $v\in \fconvf\cap C^2(\Rn)$, then for every $0\leq j\leq n$ we have
\begin{equation}
\label{eq:connection_MA_Hessian}
\binom{n}{j}\d\MA(v[j],q[n-j];x)=[\Hess v(x)]_j \d x.
\end{equation}
Here, $[\Hess v(x)]_j$ denotes the $j$th elementary symmetric function of the eigenvalues of the Hessian matrix of $v$ at $x$, where we use the convention $[\Hess v(x)]_0\equiv 1$. The right side of \eqref{eq:connection_MA_Hessian} is a \emph{Hessian measure} of $v$, and \eqref{eq:connection_MA_Hessian} shows that it naturally extends to all of $\fconvf$. See also \cite{Trudinger_Wang_Hessian_I,Trudinger_Wang_Hessian_II}. For more details on mixed Monge--Amp\`ere measures and their conjugate counterparts (see below), we refer to \cite{Colesanti-Ludwig-Mussnig-7}.

\medskip

On $\fconv$ we consider the \emph{Legendre--Fenchel transform} or \emph{convex conjugate}, which for $w\in\fconv$ is given by
$$w^*(x)=\sup\left\{ \langle x,y\rangle - w(y):  y\in\Rn \right\} $$
for $x\in\Rn$. Let us emphasize that convex conjugation is an order-reversing involution on $\fconv$ (in fact, it is essentially the only one \cite{Artstein_Milman_Legendre}), which in particular means that $(w^*)^*=w$ for every $w\in\fconv$. Under this transform, the space $\fconvf$ is dual to the space of lower semicontinuous, super-coercive convex functions,
$$\fconvs=\Bigg\{u\colon \Rn\to(-\infty,\infty] : u \text{ is l.s.c., proper, convex, and } \lim_{|x| \to \infty} \frac{u(x)}{|x|}=\infty \Bigg\}.$$
This means that $u\in\fconvs$ if and only if $u^*\in \fconvf$ (see, for example, \cite[Theorem 11.8]{RockafellarWets}). Observe that for differentiable $u\in \fconvs$ we have
\begin{equation*}
\lim\nolimits_{|x|\to\infty} |\nabla u(x)| = \infty.
\end{equation*}
 We naturally embed the space of convex bodies into $\fconvs$ by associating to each $K\in\Kn$ its \emph{convex indicator function}
$$\ind_K(x)=\begin{cases}0\quad&\text{if } x\in K,\\ \infty\quad&\text{else,} \end{cases}$$
which is the convex conjugate of the support function $h_K$.

\medskip

The Monge--Amp\`ere measure has a natural counterpart on $\fconvs$. The \emph{conjugate Monge--Amp\`ere measure} is defined as 
\begin{equation}
\label{eq:map_ma}
\MAp(u;\cdot)=\MA(u^*;\cdot)
\end{equation}
for $u\in\fconvs$ (and hence $u^*\in \fconvf$). It admits a straightforward representation in the form of
\begin{equation}
\label{eq:def_map}
\MAp(u;B)=\int_{\dom(u)} \chi_B(\nabla u(x)) \d x
\end{equation}
for Borel sets $B\subseteq \Rn$. Here, $\chi_B$ denotes the usual \emph{characteristic function} of $B$ and $\dom(u)=\{x\in\Rn : u(x) < \infty\}$ is the \emph{domain of $u$}. At this point, it is worth pointing out that a convex function is differentiable almost everywhere (w.r.t.~the $n$-dimensional Lebesgue measure) on the interior of its domain, and thus \eqref{eq:def_map} is well-defined. Let us also remark that \eqref{eq:def_map} shows that the measure $\MAp(u;\cdot)$ is the push-forward of the Lebesgue measure under the gradient of $u$.

The \emph{conjugate mixed Monge--Amp\`ere measure} of $u_1,\ldots,u_n\in\fconvs$ is defined as
\begin{equation}
\label{eq:def_mixed_map}
\MAp(u_1,\ldots,u_n;\cdot)=\MA(u_1^*,\ldots,u_n^*;\cdot).
\end{equation}
By standard properties of the Legendre--Fenchel transform, we have
\begin{equation}
\label{eq:cma_poly_exp}
\MAp\big((\lambda_1\sq u_1)\infconv \cdots \infconv (\lambda_m\sq u_m);\cdot\big)=\sum_{i_1,\ldots,i_n=1}^m \lambda_{i_1}\cdots \lambda_{i_n} \MAp(u_{i_1},\ldots,u_{i_n};\cdot)
\end{equation}
for $u_1,\ldots,u_m\in\fconvs$, $\lambda_1,\ldots,\lambda_m\geq 0$, and $m\in\N$. Here
$$(u_1\infconv u_2) (x)=\inf\left\{u_1(x-y)+u_2(y):y\in\Rn \right\}$$
for $x\in\Rn$, is the \emph{infimal convolution} or \emph{epi-sum} of $u_1$ and $u_2$. The name epi-sum is justified by the fact that
\begin{equation}
\label{eq:epi_sum}
\epi(u_1\infconv u_2)=\epi(u_1) + \epi(u_2)
\end{equation}
for $u_1,u_2\in\fconvs$, where
$$\epi w=\{(x,t)\in\Rn\times\R : w(x)\leq t\}$$
is the \emph{epi-graph} of $w\in\fconv$. Furthermore,
$$(\lambda \sq u)(x)=\lambda\,u\left(\tfrac{x}{\lambda}\right)$$
for $x\in\Rn$, denotes the \emph{epi-multiplication} of $u$ with $\lambda > 0$, with the additional convention that
$$(0\sq u)(x) = \ind_{\{o\}}(x).$$
Similar to \eqref{eq:epi_sum} the relation
$$\epi(\lambda \sq u) = \lambda \epi(u)$$
holds for $\lambda> 0$. For $\lambda=0$, we have $\epi(0\sq u)=\epi(\ind_{\{o\}})=\{(o,t)\in\R^n\times\R:t\ge 0\}$. In addition, let us note that
\begin{equation}
\label{eq:sum_conj}
(\lambda v + \mu w)^* = (\lambda \sq v^*) \infconv (\mu \sq w^*)
\end{equation}
for every $v,w\in\fconvf$ and $\lambda,\mu\geq 0$.

\medskip

For a convex function $u\in\fconvs$ and $E\in\Grass{j}{n}$ with $1\leq j \leq n$, we write
$$\proj_E u(x_E) = \min\nolimits_{y\in E^\perp} u(x_E +y)$$
with $x_E\in E$, for the \emph{projection function} of $u$, which can also be described by
\begin{equation}
\label{eq:proj_epi}
\epi \proj_E u = \proj_{E\times \lin\{e_{n+1}\}} \epi u.
\end{equation}
The following result, which can be found in \cite[Theorem 11.23]{RockafellarWets}, shows that taking convex conjugates of projections of super-coercive convex functions corresponds to restricting their conjugates to linear subspaces.

\begin{lemma}
\label{le:proj_conj_restr}
Let $1\leq j\leq n$ and $E\in\Grass{j}{n}$. For every $u\in\fconvs$ the equality
$$(\proj_E u)^*(x_E)=(u^*)\vert_E(x_E),\quad x_E\in E,$$
holds, where on the left side, convex conjugation is considered with respect to the ambient space $E$.
\end{lemma}

The next lemma is stated for the epi-sum and epi-multiplication of two functions, but the result obviously extends to finitely many functions.

\begin{lemma}
\label{le:proj_conv}
Let $1\leq j\leq n$ and $E\in\Grass{j}{n}$. If $u_1,u_2\in\fconvs$ and $\lambda_1,\lambda_2\ge 0$, then
$$
\proj_E((\lambda_1\sq u_1)\infconv (\lambda_2\sq u_2))=(\lambda_1\sq \proj_E u_1)\infconv (\lambda_2\sq \proj_E u_2),
$$
where on the right side, epi-sum and epi-multiplication are considered with respect to the ambient space $E$.    
\end{lemma}

\begin{proof}
    Combining \eqref{eq:sum_conj} and Lemma \ref{le:proj_conj_restr}, we obtain
\begin{align*}
    \left(\proj_E((\lambda_1\sq u_1)\infconv (\lambda_2\sq u_2))\right)^*&=\left((\lambda_1\sq u_1)\infconv (\lambda_2\sq u_2)\right)^*\vert_E\\
    &=\left(\lambda_1u_1^*+\lambda_2u_2^*\right)\vert_E\\
    &=\lambda_1(u_1^*\vert_E)+\lambda_2(u_2^*\vert_E)\\
    &=\lambda_1\left(\proj_E u_1\right)^*+\lambda_2\left(\proj_E u_2\right)^*\\
    &=\left((\lambda_1\sq \proj_E u_1) \infconv (\lambda_2\sq \proj_E u_2)\right)^*,
\end{align*}
where the appropriate ambient space for taking the convolution or the conjugate is clear from the context. Taking conjugates on both sides, we obtain the assertion.   
\end{proof}

When dealing with a subspace $E\in\Grass{j}{n}$, we occasionally need to restrict the measures $\MA$ and $\MAp$ (as well as their mixed versions) to functions defined on $E$, and we denote these restrictions by $\MA_E$ and $\MAp_E$, respectively.

\medskip

Next, let us address some topological properties. We equip the space $\fconv$ with the topology that is associated with \emph{epi-convergence}, where we say that a sequence $w_i\in\fconv$, $i\in\N$, epi-converges to $w\in\fconv$ if for every $x\in\Rn$,
\begin{itemize}
    \item $w(x)\leq \liminf_{i\to \infty} w_i(x_i)$ for every sequence $x_i\to x$ and
    \item $w(x)=\lim_{i\to\infty} w_i(x_i)$ for some $x_i\to x$.
\end{itemize}
This notion of convergence is induced by a metrizable topology on $\fconv$. See, for example, \cite[Theorem 7.58]{RockafellarWets}. On $\fconvf$, epi-convergence is equivalent to pointwise convergence (see \cite[Theorem 7.17]{RockafellarWets}) and on $\fconvs$, epi-convergence of $u_i$ to $u$ is equivalent to Hausdorff convergence of the level sets
$$\{x\in\Rn : u_i(x)\leq t\} \to \{x\in\Rn : u(x)\leq t\}$$
for every $t\neq \min_{x\in\Rn} u(x)$ (see \cite[Lemma 5]{Colesanti-Ludwig-Mussnig_ValConv}). Note that for $t<\min_{x\in\Rn} u(x)$, the level set $\{x\in\Rn : u(x)\leq t\}$ is empty and we say that a sequence of convex sets $C_i$, $i\in\N$, converges to the empty set if $C_i=\emptyset$ for every $i$ large enough.

Lastly, let us remark that $w\mapsto w^*$ is continuous on $\fconv$ (see \cite[Theorem 11.34]{RockafellarWets}).

\section{Supports of Mixed Area Measures}
\label{se:supports}

\subsection{Kubota-type Formulas}
\label{se:kubota}
We begin this section with a simple but useful consequence of Lemma~\ref{le:char_mean_width}. For $k \in \N$ we denote by $\hm^k$ the $k$-dimensional Hausdorff measure.
\begin{lemma}
\label{le:int_sn_ball}
If $\psi\colon \Kn\to\R$ is Minkowski additive and continuous, then
$$\int_{\sn} \psi([o,e]) \d\hm^{n-1}(e)=\kappa_{n-1} \psi(B^n).$$
\end{lemma}
\begin{proof}
Since the statement is trivial for $n=1$, we may assume $n\geq 2$. Observe that the map $\varphi\colon \Kn\to\R$, defined by
$$\varphi(K)= \int_{\SO(n)} \psi(\vartheta K) \d\vartheta, \quad K\in\Kn,$$ satisfies the assumptions of Lemma~\ref{le:char_mean_width}, where the integration is with respect to the Haar probability measure on $\SO(n)$. Thus, there exists $c\in\R$ such that $\varphi = c V_1$. Since $\varphi(B^n)=\psi(B^n)$ and $V_1(B^n)=\frac{n\kappa_n}{\kappa_{n-1}}$, we conclude that $c=\frac{\kappa_{n-1}}{n\kappa_n}\psi(B^n)$. We therefore obtain
$$\int_{\sn} \psi([o,e]) \d\hm^{n-1}(e) = n \kappa_n \int_{\SO(n)} \psi(\vartheta ([o,e_1])) \d \vartheta = n\kappa_n\, \varphi([o,e_1]) = \kappa_{n-1} \psi(B^n),$$
which proves the assertion.
\end{proof}

Recall that $\ell=\operatorname{span}\{e_n\}$ and $L=\ell^\perp\cong\R^{n-1}$. For $1\leq k\leq n$ we consider the Grassmannian subspaces 
$$\Gr(\ell,k)=\{E\in \Grass{k}{n} : \ell \subseteq E\}$$
and
$$\Gr(F,j)=\{E\in \Grass{j}{n} : E \subseteq F\},$$
for $F\in\Gr(n,k)$ and $1\leq j \leq k$.
Since $\Rn=L+\ell$ (which is a direct sum of linear subspaces), we can rewrite
$$\Gr(\ell,k)=\{F + \ell : F\in \Gr(L,k-1)\}.$$
Integration on $\Gr(\ell,k)$ will always be understood with respect to the unique probability measure that is invariant under those rotations that fix $\ell$. Similarly, integration on $\Gr(F,j)$ will always be understood with respect to the Haar probability measure that is invariant under rotations that preserve $F$.

We use standard properties of integration on the Grassmannian to see that
\begin{align}
\label{eq:int_grass_n-1_n_e_n}
\int_{\Gr(\ell,n-1)} f(E) \d E& = \int_{\Gr(L,n-2)} f(\tilde{E}+ \ell) \d \tilde{E}\notag\\
&= \frac{1}{(n-1)\kappa_{n-1}} \int_{\s_L^{n-2}} f(e^\perp)\d\hm^{n-2}(e)
\end{align}
for measurable functions $f\colon\Gr(\ell,n-1)\to [0,\infty)$, where $\s_L^{n-2}=\s^{n-1}\cap L$. Similarly,
\begin{align}
\begin{split}
\label{eq:int_grass_ell_split}
\int_{\Gr(\ell,k)} g(E) \d E &= \int_{\Gr(L,k-1)} g(\tilde{E}+\ell) \d\tilde{E}\\
&= \int_{\Gr(L,k)} \int_{\Grass{k-1}{F}} g(\tilde{E}+\ell) \d\tilde{E}\d F\\
&= \int_{\Gr(\ell,k+1)} \int_{\Gr(\ell,F,k)} g(E)\d E\d F
\end{split}
\end{align}
for $1\leq k \leq n-1$ and measurable functions $g\colon\Gr(\ell,k)\to [0,\infty)$. Here, $\Gr(\ell,F,k)$ denotes the space of $k$-dimensional linear subspaces of $F$ that contain $\ell$, and integration is understood with respect to the probability measure that is invariant under rotations that map both $\ell$ and $F$ onto themselves. See, for example, formula (7.4) in \cite{schneider_weil}.

\begin{proof}[Proof of Theorem~\ref{thm:ck_sam}]
For $j=n-1$, the statement is trivial. Thus, we assume that $1\le j\le n-2$. Let $F\in \Gr(\ell,k+1)$ for some $2\leq k \leq n-1$ (we will have $k\ge j+1$). For the proof, we can assume that $f$ is continuous and remark that the general case follows by standard arguments. 
We use \eqref{eq:int_grass_n-1_n_e_n} with respect to the $(k+1)$-dimensional ambient space $F$ to obtain
\begin{align}
\label{eq:ck_proof_1}
&\int_{\Gr(\ell,F,k)} \int_{\s_E^{k-1}} f (z) \d S_E(\proj_E K_1,\ldots,\proj_E K_{k-1},z)\d E\notag\\
&=\int_{\Gr(\ell,F,k)} \int_{\s_F^k \cap E} f(z)\d S_E(\proj_E \proj_F K_1,\ldots,\proj_E \proj_F K_{k-1},z)\d E\notag\\
&= \frac{1}{k \kappa_k} \int_{\s_F^k\cap L} \int_{\s_F^k \cap e^\perp} f(z)\notag\\
&\qquad\qquad \d S_{F\cap e^\perp }(\proj_{e^\perp} \proj_F K_1,\ldots, \proj_{e^\perp} \proj_F K_{k-1},z)  \d \hm^{k-1}(e)
\end{align}
for arbitrary $K_1,\ldots,K_{k-1}\in \Kn$. By Lemma~\ref{le:mx_area_meas_proj} we have
\begin{align}
\label{eq:ck_proof_2}
&\int_{\s_F^k\cap e^\perp} f(z) \d S_{F\cap e^\perp }(\proj_{e^\perp} \proj_F K_1,\ldots, \proj_{e^\perp} \proj_F K_{k-1},z)\nonumber\\
&=k \int_{\s_F^k} f(z)\d S_F(\proj_F K_1,\ldots, \proj_F K_{k-1},[o,e],z)
\end{align}
for every $e\in \s_F^k\cap L$. Combining \eqref{eq:ck_proof_1} and \eqref{eq:ck_proof_2} together with Lemma~\ref{le:int_sn_ball} (applied in $F\cap L\in \Gr(n,k)$ and using Lemma~\ref{le:properties_int_s}), we obtain
\begin{align*}
&\int_{\Gr(\ell,F,k)} \int_{\s_E^{k-1}} f(z) \d S_E(\proj_E K_1,\ldots,\proj_E K_{k-1},z)\d E\\
&=\frac{1}{\kappa_k} \int_{\s_F^k\cap L} \int_{\s_F^k} f(z)\d S_F(\proj_F K_1,\ldots,\proj_F K_{k-1},[o,e],z)\d\hm^{k-1}(e)\\
&=\frac{\kappa_{k-1}}{\kappa_k} \int_{\s_F^k} f(z)\d S_F(\proj_F K_1,\ldots,\proj_F K_{k-1},\proj_F B_L^{n-1},z).
\end{align*}
We now use \eqref{eq:int_grass_ell_split} and apply the last equality recursively to obtain
\begin{align*}
&\frac{1}{\kappa_j} \int_{\Gr(\ell,j+1) } \int_{\s_E^{j}} f(z) \d S_E(\proj_E K_1,\ldots,\proj_E K_j,z)\d E\\
&=\frac{1}{\kappa_j}\int_{\Gr(\ell,j+2)} \int_{\Gr(\ell,F,j+1)} \int_{\s_E^{j}} f(z) \d S_E(\proj_E K_1,\ldots,\proj_E K_j,z)  \d E \d F\\
&= \frac{1}{\kappa_{j+1}} \int_{\Gr(\ell,j+2)} \int_{\s_F^{j+1}} f(z) \d S_F(\proj_F K_1,\ldots,\proj_F K_j, \proj_F B_L^{n-1},z) \d F\\
&\;\;\vdots\\
&=\frac{1}{\kappa_{n-1}} \int_{\s^{n-1}} f(z) \d S(K_1,\ldots,K_j,B_L^{n-1}[n-1-j],z)
\end{align*}
for every $1\leq j <n-1$ and $K_1,\ldots,K_j\in\Kn$, which completes the proof.
\end{proof}

\begin{remark} 
Observe that the classical Cauchy--Kubota formulas for intrinsic volumes of a convex body $K\in\Kn$ (see, for example, formulas (5.71) and (5.72) in \cite{Schneider_CB}) can be retrieved from Theorem~\ref{thm:ck_sam}. Indeed, if $K$ is embedded into $\R^{n+1}$ via the identification $\R^n=e_{n+1}^\perp$, then $S_n(K,\{e_{n+1}\})=V_n(K)$ and one simply needs to apply Theorem~\ref{thm:ck_sam} with respect to the ambient space $\R^{n+1}$.
\end{remark}

\begin{remark}
An alternative proof of Theorem~\ref{thm:ck_sam} can be based on a special case of \cite[Theorem 8]{HTW11} in combination with a spherical Blaschke--Petkantschin formula \cite[Lemma 5.3]{HT19}. The implications of such an approach will be discussed in future work. 
\end{remark}

\subsection{A First Description of Supports}
\label{se:desc_supp}
Let $K\in\Kn$. For $x \in \partial K$, the boundary of $K$, we denote by $N(K,x)$ the \emph{normal cone} of $K$ at $x$, that is
$$N(K,x)=\{z\in\Rn : \langle x,z \rangle = h_K(z)\}.$$
The elements of $N(K,x)\setminus\{o\}$ are the \emph{outer normal vectors} of $K$ at $x$.
In addition, let
$$F(K,z)=\{x\in K : \langle x,z \rangle = h_K(z)\}$$
denote the \emph{support set} of $K$ with outer normal vector $z\in \Rn\backslash\{o\}$. In particular, the point $z$ is an outer normal vector of $K$ at each point of $F(K,z)$.
For a non-empty convex subset $F$ of $K$, we define $N(K,F)= N(K,x)$, where $x\in\operatorname{relint} F$, which is not dependent on the particular choice of $x$.  We define the \emph{touching cone} of $K$ at $z$, denoted by $T(K,z)$, as the unique face of $N(K,F(K,z))$ that contains $z$ in its relative interior.

For $0\leq r \leq n-1$, we say that $z$ is an \emph{$r$-extreme normal vector of $K$} if $\dim T(K,z)\leq r+1$. This is equivalent to the fact that there do not exist $r+2$ linearly independent normal vectors $z_1,\ldots,z_{r+2}$ at the same boundary point of $K$ such that $z=z_1+\cdots + z_{r+2}$. A $0$-extreme normal vector is simply called \emph{extreme}.

We need the following result, which is due to Schneider \cite[Satz 4]{Schneider_75}. It was previously conjectured by Weil \cite{weil_73}, who gave a proof for the case $j=1$. See also \cite[Theorem 4.5.3]{Schneider_CB}, where this result is proved under the assumption that $K$ has non-empty interior. We are only interested in the case $j=n-1$ and remark that in this case, it is straightforward to lift the restriction on the dimension of $K$, since the description of $S_{n-1}(K,\cdot)$ is trivial when $\dim K<n$.

\begin{theorem}
	\label{thm:supp_s_j}
	Let $K\in\Kn$ and let $0\leq j\leq n-1$. The support of $S_j(K,\cdot)$ is the closure of the set of all $(n-1-j)$-extreme unit normal vectors of $K$.
\end{theorem}

\begin{lemma}
\label{le:int_s_cont_grass}
Let $K\in\Kn$ and let $1\leq j< k \leq n$. If $f\in C(\s^{n-1})$, then
\begin{align*}
    I_{f,K} \colon \Gr(\ell,k) &\to \R\\
    E &\mapsto \int_{\s_E^{k-1}} f(z_E) \d S_E(\proj_E K[j],\proj_E B_L^{n-1}[k-1-j],z_E)
\end{align*}
is continuous on $\Gr(\ell,k)$.
\end{lemma}
\begin{proof}
If $k=n$, there is nothing to show, and we may therefore assume $k<n$. Throughout the following let $\bar E\in\Gr(\ell,k)$ be fixed and let $E_i\in\Gr(\ell,k)$, $i\in\N$, be a sequence converging to $\bar E$. We can find a sequence $\vartheta_i\in\SO(n)$, $i\in\N$, such that $\vartheta_i x\to x$ for every $x\in\Rn$ as $i\to\infty$ and such that $\{\vartheta_i x\colon x\in E_i\}=\bar E$ for every $i\in\N$. We now have
\begin{align}\label{eq:eqrot}
I_{f,K}(E_i)&=\int_{\s_{E_i}^{k-1}} f(z_{E_i}) \d S_{E_i}(\proj_{E_i} K[j],\proj_{E_i} B_L^{n-1}[k-1-j],z_{E_i})\nonumber\\
&=\int_{\s_{\bar E}^{k-1}} (f\circ\vartheta_i^{-1})(z_{\bar E}) \d S_{\bar E}(\proj_{\bar E} (\vartheta_i K)[j], \proj_{\bar E} B_L^{n-1} [k-1-j],z_{\bar E})\nonumber\\
&=I_{f_i,K_i}(\bar E),
\end{align}
where we write $f_i=f\circ \vartheta_i^{-1}$ and $K_i= \vartheta_i K$ for $i\in\N$. Note that $K_i$ converges to $K$ and $\proj_{\bar E} K_i$ converges to $\proj_{\bar E} K$ as $i\to\infty$.

It follows from Lemma~\ref{le:properties_int_s} that
$$S_{\bar E}(\proj_{\bar E} K_i[j],\proj_{\bar E} B_L^{n-1} [k-1-j],\s_{\bar E}^{k-1})$$
is a convergent sequence in $i$ and thus it is bounded by a constant $C>0$ for every $i\in\N$.

Now let $\varepsilon>0$ be arbitrary. Since $\sn$ is compact and $f$ is continuous, there exists $i_1\in\N$ such that
$$
\big|f(z_{\bar E})-f_i(z_{\bar E})\big|\leq \frac{\varepsilon}{2C}
$$
for $z_{\bar E}\in \s_{\bar E}^{k-1}$ and $i\geq i_1$. We now have
\begin{align}
\label{eq:I_est_1}
&\big\vert I_{f,K_i}(\bar E)-I_{f_i,K_i}(\bar E)  \big\vert \nonumber\\
&\leq \int_{\s_{\bar E}^{k-1}} \big\vert f(z_{\bar E})-f_i(z_{\bar E})\big\vert \d S_{\bar E}(\proj_{\bar E} K_i[j],\proj_{\bar E} B_L^{n-1}[k-1-j],z_{\bar E})\nonumber\\
&\leq \frac{\varepsilon}{2}
\end{align}
for $i\geq i_1$. Furthermore, by Lemma~\ref{le:properties_int_s} there exists $i_2\in\N$ such that
\begin{equation}
\label{eq:I_est_2}
\left\vert I_{f,K}(\bar E)-I_{f,K_i}(\bar E) \right\vert \leq \frac{\varepsilon}{2}
\end{equation}
for $i\geq i_2$. Thus, combining \eqref{eq:eqrot}, \eqref{eq:I_est_1} and \eqref{eq:I_est_2} gives
$$\left\vert I_{f,K}(\bar{E})-I_{f,K}(E_i)\right\vert \leq \left\vert I_{f,K}(\bar E)-I_{f,K_i}(\bar E) \right\vert+ \left\vert I_{f,K_i}(\bar E)-I_{f_i,K_i}(\bar E)\right\vert\leq \varepsilon$$
for $i\geq \max\{i_1,i_2\}$, which completes the proof.
\end{proof}

For a linear subspace $E$ of $\Rn$ and a convex body $K\subset E$, we write $\ext_E(K)$ for the set of extreme unit normal vectors of $K$ with respect to the ambient space $E$.

\begin{theorem}
\label{thm:descrp_supp_s_proj}
If $1\leq j \leq n-1$ and $K\in\Kn$, then
\begin{equation}
\label{eq:supprelation}
    \supp \Sdisk_j(K,\cdot) =\cl\left\{z\in \sn:\exists E\in\Gr(\ell,j+1)\text{ \rm such that }z\in \ext_E(\proj_E K)\right\}.
\end{equation}
\end{theorem}
\begin{proof}
     Suppose that $z\in \sn$ is not in the support of $\Sdisk_j(K,\cdot)$. We will show that then, for every $E\in\Gr(\ell,j+1)$ such that $z\in E$, the vector $z$ does not belong to   $\cl \ext_E (\proj_E K)$. Once this is proved, the fact that the support of a measure is closed  implies that 
    \begin{equation*}
    \supp \Sdisk_j(K,\cdot) \supseteq \cl\left\{z\in \sn:\exists E\in\Gr(\ell,j+1)\text{ \rm such that }z\in \ext_E(\proj_E K)\right\}.
    \end{equation*}
    If $z \notin \supp \Sdisk_j(K,\cdot)$, then there exist  
    an open (spherical) neighborhood $A$ of $z$ and a non-negative function $f \in C(\sn)$ such that $f(z)>0$ on $A$ and
    $$
    \int_{\sn} f(w) \d \Sdisk_j(K,w)=0.
    $$
Theorem \ref{thm:ck_sam} implies that
    $$
    \int_{\Gr(\ell,j+1)} \int_{\s_E^j} f(w_E) \d S_E(\proj_E K[j], w_E)\d E =0.
    $$
    By Lemma \ref{le:int_s_cont_grass} and since $f$ is continuous and non-negative, this is true if and only if 
    $$
    \int_{\s_E^j} f(w_E) \d S_E(\proj_E K[j], w_E) =0
    $$
    for every $E \in \Gr(\ell,j+1)$. In particular, this holds for every $E\in \Gr(\ell,j+1)$ such that $z \in E$. Since $f\vert_{\s_E^j}$ is continuous and non-negative, it follows that
    $$
     \{ f\vert_{\s_E^j}>0\} \cap \supp S_E(\proj_E K[j], \cdot)= \emptyset,
    $$
    which in turn implies that $z \notin \supp S_E(\proj_E K[j], \cdot)$, since $f\vert_{\s_E^j}$ is positive in a neighborhood of $z\in E$ relative to $E$. The required assertion is now implied by Theorem~\ref{thm:supp_s_j}, applied to $\proj_E K$ in the ambient space $E$. 

For the reverse inclusion, first note that 
\begin{align}\label{eq:refThm3.8new}
    &\cl\left\{z\in \sn:\exists  E\in\Gr(\ell,j+1)\text{ \rm such that }z\in \ext_E(\proj_E K)\right\}\nonumber\\
    &= \cl \left(\bigcup_{E \in \Gr(\ell,j+1)} \supp S_E(\proj_E K[j],\cdot)\right).\
\end{align}
Suppose that $z\in\sn$ does not belong to this set. In this case, there is some $\tau>0$ such that the closed geodesic ball $B_\tau(z)\subset \sn$ centered at $z$ with radius $\tau$ is disjoint from the set. Let $U_\tau(z)$ denote the relative interior of $B_\tau(z)$. Theorem \ref{thm:ck_sam} yields
\begin{equation*}
\int_{\sn}\chi_{U_\tau(z)}(w)\d \Sdisk_j(K,w)= \int_{\Gr(\ell,j+1)} \int_{\s_E^j} \chi_{U_\tau(z)}(w_E) \d S_E(\proj_E K[j], w_E)\d E =0, 
\end{equation*}
hence $z\notin \supp \Sdisk_j(K,\cdot)$. 
\end{proof}

\begin{remark}
\label{re:descrp_supp_s_proj}
Theorem~\ref{thm:descrp_supp_s_proj} and relation  \eqref{eq:refThm3.8new} imply that 
$$
\supp \Sdisk_j(K,\cdot)=
\cl \left(\bigcup_{E \in \Gr(\ell,j+1)} \supp S_E(\proj_E K[j],\cdot)\right).
$$
\end{remark}

\subsection{Proofs of Theorem~\ref{thm:supp_ext} and Corollary~\ref{cor:supp_s_nested}}
\label{se:proof_thm_supp_ext}
For convex bodies $K_1,\ldots,K_{n-1}\in\Kn$ we say that $z\in\sn$ is \emph{$(K_1,\ldots,K_{n-1})$-extreme} if there exist hyperplanes $E_1,\ldots,E_{n-1}\in \Gr(n,n-1)$ such that $T(K_i,z)\subseteq E_i$ for $1\leq i\leq n-1$ and such that
$$\dim (E_1 \cap \ldots \cap E_{n-1})=1.$$
We denote the set of all such vectors by $\ext(K_1,\ldots,K_{n-1})$.

We write $TS(K,z)=T(K,z)^\perp$ for the corresponding \emph{touching spaces}. In terms of touching spaces  we can rephrase the condition  $z\in \ext(K_1,\ldots,K_{n-1})$ by requiring the existence of vectors $w_1,\ldots,w_{n-1}\in\sn$ such that $w_i\in TS(K_i,z)$ for every $1\leq i\leq n-1$ and such that
$$\dim(\lin\{w_1,\ldots,w_{n-1}\})=n-1.$$
Note that we always have $TS(K_i,z)\subseteq z^\perp$ and thus this means that $w_1,\ldots,w_{n-1}$ are linearly independent and 
$$\lin\{w_1,\ldots,w_{n-1}\}=z^\perp.$$
Furthermore, $z$ is an extreme vector of $K$ if and only if $\dim TS(K,z) = n-1$.

\medskip

In order to understand $\ext(K[j],B_L^{n-1}[n-1-j])$, we first describe the touching space $TS(B_L^{n-1},z)$. We distinguish two cases, where $\pos$ denotes the \emph{positive hull}.
\begin{description}
	\item $z\in\{\pm e_n\}$: In this case $T(B_L^{n-1},z)=\pos\{z\}$. Therefore, $TS(B_L^{n-1},z)=e_n^\perp = L$. In particular, $\dim TS(B_L^{n-1},z)=n-1$.
    \item $z\notin\{\pm e_n\}$: In this case $T(B_L^{n-1},z)=\pos\{z,e_n,-e_n\}$. Thus, $TS(B_L^{n-1},z)=z^\perp \cap e_n^\perp=z^\perp \cap L$, which is an $(n-2)$-dimensional linear subspace.
\end{description}

Our observations lead us to the following description. Lemma \ref{le:descr_ext} remains clearly true for $j=n-1$. Since the case distinction as to whether $z$ is in $\{\pm e_n\}$ or not is not needed in this case, we did not include $j=n-1$ in the statement of the lemma.

\begin{lemma}
\label{le:descr_ext}
Let $1\leq j\leq n-2$ and $K\in\Kn$. For $z\in\{\pm e_n\}$ we have
$$z\in\ext(K[j],B_L^{n-1}[n-1-j])$$
if and only if $\dim TS(K,z)\geq j$. For $z\in\sn\backslash\{\pm e_n\}$ we have
$$z\in\ext(K[j],B_L^{n-1}[n-1-j])$$
if and only if $\dim TS(K,z)\geq j$ and $ TS(K,z)\not\subseteq L$.
\end{lemma}
\begin{proof}
We start with the case $z\in\{\pm e_n\}$. Since $TS(B_L^{n-1},z)=L$,
this means that
$$z\in \ext(K[j],B_L^{n-1}[n-1-j])$$
if and only if there exist vectors $w_1,\ldots,w_j\in TS(K,z)$ and $w_{j+1},\ldots,w_{n-1}\in L$ such that $\lin\{w_1,\ldots,w_{n-1}\}=L$. Considering that $TS(K,z)\subseteq L$ and $\dim L = n-1$, the above is equivalent to $\dim TS(K,z)\geq j$. 

Now we consider the case $z\in\sn\setminus\{\pm e_n\}$. By our description of $TS(B_L^{n-1},z)$ above this means that $z$ is $(K[j],B_L^{n-1}[n-1-j])$-extreme if and only if there exist linearly independent vectors $w_1,\ldots,w_j\in TS(K,z)\subseteq z^\perp$ and $w_{j+1},\ldots,w_{n-1}\in z^\perp \cap L$ such that $\lin\{w_1,\ldots,w_{n-1}\}=z^\perp$. If the latter holds, then we clearly have $\dim TS(K,z)\geq j$, and $ TS(K,z)\not\subseteq L$, since otherwise $w_1,\ldots,w_{n-1}\in z^\perp\cap L$ contradicts $\dim (z^\perp\cap L)=n-2$.  
Now assume that $\dim TS(K,z)\geq j$ and $ TS(K,z)\not\subseteq L$. Setting $k=\dim TS(K,z)$, we have
$$\dim (TS(K,z)\cap L)= \dim TS(K,z)-1=k-1\ge j-1$$
and $\dim(z^\perp\cap L)=n-2$. Therefore, we can choose linearly independent vectors\linebreak$w_2,\ldots,w_{k}\in TS(K,z)\cap L\subseteq z^\perp\cap L$, $w_{k+1},\ldots,w_{n-1}\in z^\perp\cap L$, and $w_1\in TS(K,z)\setminus L$. Since $k\ge j$ we have $w_1,\ldots,w_j\in TS(K,z)$ and $w_{j+1},\ldots,w_{n-1}\in z^\perp\cap L$. 
\end{proof}

Before we can continue to prove our main result of this section, we need the following lemma, which is a consequence of \cite[Lemma 3.3]{Hug_Reichert_support}.
\begin{lemma}
\label{le:dim_T}
Let $1\leq j\leq n-1$, $E\in\Gr(\ell,j+1)$ and $K\in\Kn$. If $z\in E$, then
$$\dim \proj_E TS(K,z)=\dim TS_E(\proj_E K,z),$$
where $TS_E(\proj_E K,z)$ denotes the touching space of $\proj_E K$ at $z$ with respect to the ambient space $E$.
\end{lemma}

\begin{lemma}\label{lem:ext=set}
If $1\le j\le n-2$ and $K\in\Kn$, then 
\begin{align}
\label{eq:sets_equal}
\ext(K[j],&B_L^{n-1}[n-1-j])\nonumber\\
&= \{z\in\sn : \exists\,E\in \Gr(\ell,j+1) \text{ such that } z\in \ext_E(\proj_E K)\}.
\end{align}
\end{lemma}
\begin{proof}
First, let us observe that by Lemma~\ref{le:dim_T} a vector $z$ is an element of the right side of \eqref{eq:sets_equal} if and only if there exists $E\in \Gr(\ell,j+1)$ such that
\begin{equation}
\label{eq:dim_proj_TS_j}
z\in E \quad \text{and}\quad  j=\dim TS_E(\proj_E K,z)=\dim \proj_E TS(K,z).
\end{equation}
Here we have used that $\dim E = j+1$.

We start with the case $z\in\{\pm e_n\}$. Lemma~\ref{le:descr_ext} shows that the vector $z$ is\linebreak$(K[j],B_L^{n-1}[n-1-j])$-extreme if and only if $\dim TS(K,z)\geq j$. Since $TS(K,z)\subseteq L$ and $\ell \perp L$, this is equivalent to the existence of $E\in \Gr(\ell,j+1)$ such that
$$j=\dim \proj_E TS(K,z).$$
Since trivially $z\in \ell\subset E$, we have thus shown equivalence with \eqref{eq:dim_proj_TS_j}.

Now, let $z\in\sn\setminus\{\pm e_n\}$. Again by Lemma~\ref{le:descr_ext}, $z\in \ext(K[j],B_L^{n-1}[n-1-j])$ if and only if  $\dim TS(K,z)\geq j$ and $  TS(K,z)\not\subseteq L$. If these two conditions are satisfied, then
$$\dim (TS(K,z)\cap L)=\dim TS(K,z)-1\ge j-1$$
and
$$\lin \{z,e_n\}=(z^\perp\cap e_n^\perp)^\perp\subseteq(TS(K,z)\cap L)^\perp.$$
Let $U$ be a linear subspace of $TS(K,z)\cap L$ of dimension $j-1$.  The linear subspace $E=\lin\{z,e_n\}+U$ satisfies $z\in E$, $\ell\subset E$ and $\dim E=j-1+2=j+1$. Since $U\subseteq E\cap TS(K,z)$, we have $U\subseteq\proj_E TS(K,z)$, so that
$$\dim\proj_E TS(K,z)\ge \dim U = j-1.$$
Since $TS(K,z)\subseteq z^\perp$, we have $\proj_E TS(K,z)\subseteq \ell+U$ and thus
$$\dim\proj_E TS(K,z)\le j.$$
Suppose that $\dim\proj_E TS(K,z)= j-1$. By our choice of $U$ this implies
$$\proj_E TS(K,z)=U\subseteq TS(K,z)\cap L$$
and in particular, $\dim(\proj_\ell TS(K,z))=0$. Therefore $TS(K,z)\subseteq  \lin\{z,e_n\}^\perp\subseteq L$, which is a contradiction. Thus we have found a subspace $E\in \Gr(\ell,j+1)$ with $z\in E$ and $\dim\proj_E TS(K,z)= j$, which gives \eqref{eq:dim_proj_TS_j}.  

Conversely, assume that there is some $E\in \Gr(\ell,j+1)$ with $z\in E$ and such that $\dim\proj_E TS(K,z)= j$, which clearly implies $\dim TS(K,z)\ge j$. Suppose that $TS(K,z)\subseteq L$. We now have
$$TS(K,z)\subseteq L\cap z^\perp=e_n^\perp\cap z^\perp=\lin\{e_n,z\}^\perp.$$
Since $E=\lin\{e_n,z\}+U$ with some linear subspace $U\subseteq \lin\{e_n,z\}^\perp$, it follows that $\proj_E TS(K,z)\subseteq U$, which is a contradiction, since $\dim U=j-1$. Thus, we must have $TS(K,z)\not\subseteq L$. 
\end{proof}

\begin{proof}[Proof of Theorem~\ref{thm:supp_ext}]
Since $\ext(K[n-1])=\ext K$, the case $j=n-1$ follows from Theorem~\ref{thm:supp_s_j}. If $1\le j\le n-2$, then the assertion is implied by a combination of  Theorem~\ref{thm:descrp_supp_s_proj} and Lemma~\ref{lem:ext=set}.
\end{proof}

\begin{proof}[Proof of Corollary~\ref{cor:supp_s_nested}]
First, let us remark that Lemma~\ref{le:descr_ext} holds trivially for $j=n-1$, since $z\in \ext (K[n-1])$ if and only if $\dim TS(K,z)=n-1$. Hence, an application of that lemma yields
$$
\ext(K[k],B_L^{n-1}[n-1-k])\subseteq \ext(K[j],B_L^{n-1}[n-1-j])
$$
for $1\le j\le k\le n-1$. Now the assertion is a consequence of Theorem \ref{thm:supp_ext}.
\end{proof}

\section{New Properties of Mixed Monge--Amp\`ere Measures}
\label{se:properties_ma}

\subsection{Connections with Mixed Area Measures and Mixed Volumes\texorpdfstring{\\}{ }on \texorpdfstring{$\K^{n}$}{Kn}}
\label{se:connections_ma}
For $\omega\subseteq \sn$ we write
$$\widehat\omega =\{tu : t\in[0,1], u\in \omega\}\subset \Rn$$
and
$$\widetilde\omega = \{tu : t\in (0,1), u\in \omega\}\subset \Rn.$$
If $\omega$ is a Borel set, then so are $\widehat\omega$ and $\widetilde\omega$. 
Throughout the following let $q(x)=|x|^2/2$ for $x\in\Rn$. The next result is a consequence of \cite[Corollary 5.10]{Colesanti-Hug_MM}.

\begin{lemma}
\label{le:sam_ma_hat}
If $K\in\Kn$ and $\omega\subseteq \sn$ is a Borel set, then
$$S_{n-1}(K,\omega)=n\,\MA(h_{K}[n-1],q;\widehat\omega).$$
\end{lemma}

We establish the following generalization.

\begin{corollary}\label{Cor:mamMA}
If $K_1,\ldots,K_{n-1}\in\Kn$ and $\omega\subseteq \sn$ is a Borel set, then
$$S(K_1,\ldots,K_{n-1},\omega)=n\,\MA(h_{K_1},\ldots,h_{K_{n-1}},q;\widehat\omega).$$
\end{corollary}
\begin{proof}
Let $m\in\N$, $K_1,\ldots,K_m\in\Kn$, and $\lambda_1,\ldots,\lambda_m\geq 0$. As an immediate consequence of \eqref{eq:ma_linear}, we have
\begin{align}
\label{eq:PLdetail4.3}
\MA(&h_{\lambda_1 K_1 +\cdots + \lambda_m K_m}[j],q[n-j];\eta)\notag\\
&= \sum_{i_1,\ldots,i_{j}=1}^m \lambda_{i_1} \cdots \lambda_{i_{j}} \MA(h_{K_{i_1}},\ldots,h_{K_{i_{j}}},q[n-j];\eta)
\end{align}
for $1\leq j\leq n$ and Borel sets $\eta\subseteq\Rn$. 
An application of \eqref{eq:PLdetail4.3} for $j=n-1$ and Lemma~\ref{le:sam_ma_hat} together yield
\begin{align*}
S_{n-1}(\lambda_1 K_1 + \cdots + \lambda_m K_m, \omega) &= n\, \MA(h_{\lambda_1 K_1 + \cdots + \lambda_m K_m}[n-1],q;\widehat\omega)\\
&= \sum_{i_1,\ldots,i_{n-1}=1}^m \lambda_{i_1}\ldots\lambda_{i_{n-1}} n\,\MA(h_{K{i_1}},\ldots,h_{K_{i_{n-1}}},q;\widehat\omega)
\end{align*}
for Borel sets $\omega\subseteq \sn$. The result now follows after a comparison with relation \eqref{eq:def_mixed_s}.
\end{proof}

Next, we turn our attention to mixed volumes. For the proof of the following result, we use that for $K\in\Kn$ and $r>0$ the function $\ind_K\infconv (r\sq q)\in\fconvs$ is continuously differentiable on $\Rn$ with
\begin{equation}
\label{eq:grad_moreau_env}
\nabla (\ind_K \infconv (r\sq q))(x)=
\frac 1r (x-p_K(x))
\end{equation}
for $x\in\Rn$. See, for example \cite[Theorem 2.26 (b)]{RockafellarWets} (or use $\ind_K \infconv (r\sq q))=(2r)^{-1} d_K^2$ and apply \cite[Proposition 5.3]{Colesanti-Hug_MM}). Here $p_K\colon \Rn\to K$ denotes the \emph{metric projection} of $K$, which assigns to $x\in\Rn$ the unique point $p_K(x)\in K$ such that $|p_K(x)-x|=\min\{|y-x|:y\in K\}$, which is the distance $d_K(x)$ from $x$ to $K$.

\begin{lemma}
\label{le:ma_vanish_o_s}
For $K\in\Kn$ and $0\leq j\leq n-1$,
$$\MA(h_K[j],q[n-j];\{o\}\cup \s^{n-1})=0.$$
\end{lemma}
\begin{proof}
For $K\in\Kn$ and $r>0$ it follows from \eqref{eq:grad_moreau_env} that $\nabla (\ind_K\infconv (r\sq q))(x)= o$ if and only if $x\in K$. Furthermore, $\nabla (\ind_K\infconv (r\sq q))(x)\in\sn$ if and only if $|x-p_K(x)|=r$, which is the case precisely when $x\in\partial (K+ r B^n)$. Using \eqref{eq:map_ma}, \eqref{eq:def_map}, and \eqref{eq:sum_conj}, we thus conclude that
\begin{align*}
\MA(h_K+r q;\{o\}\cup \s^{n-1}) &= \MAp\big(\ind_K \infconv (r\sq q);\{o\}\cup \s^{n-1}\big)\\
&=\int_{\Rn} \chi_{\{o\}\cup \s^{n-1}}\big(\nabla (\ind_K \infconv (r\sq q))(x)\big) \d x\\
&=\int_{\Rn} \chi_{K}(x)\d x=V_n(K).
\end{align*}
Since $\ind_K\infconv(0\sq q)=\ind_K \infconv \ind_{\{o\}} = \ind_K$, the above also holds for $r=0$. In particular, $\MA(h_K+r q;\{o\}\cup \s^{n-1})$ is constant in $r$. On the other hand, it is a consequence of \eqref{eq:mixed_ma_multinom} that
$$\MA(h_K+ rq;\{o\} \cup \s^{n-1}) = \sum_{j=0}^n \binom{n}{j} r^{n-j} \MA(h_K[j],q[n-j];\{o\}\cup \s^{n-1}),$$
and the conclusion now follows after a comparison of coefficients.
\end{proof}

\begin{remark}
If $0\leq j\leq n-1$ and $K_1,\ldots,K_j\in\K^n$, then Lemma~\ref{le:ma_vanish_o_s} and \eqref{eq:PLdetail4.3} imply that 
$$\MA(h_{K_1},\ldots,h_{K_j},q[n-j];\{o\}\cup \s^{n-1})=0.$$
In particular, this shows that in the following theorem, the integral is well defined since the integration can be extended over $\interior (B^n)\setminus\{o\}$ or over $B^n\setminus\{o\}$ without changing the integral, where $\interior (B^n)$ denotes the interior of the Euclidean unit ball.
\end{remark}

\begin{theorem}
If $K_1,\ldots,K_n\in\Kn$, then
$$V(K_1,\ldots,K_n)=\int_{\interior(B^n)} h_{K_n}\left(\frac{x}{|x|}\right) \d\MA(h_{K_1},\ldots,h_{K_{n-1}},q;x).$$
\end{theorem}
\begin{proof}
For any $K\in\Kn$ and Borel set $\omega\subseteq \s^{n-1}$, it follows from Lemma~\ref{le:sam_ma_hat} hat
\begin{align*}
S_{n-1}(K,\omega)&= n \MA(h_K[n-1],q;\widehat\omega)\\
&= n \MA(h_K[n-1],q;\widetilde \omega)\\
&= n \int_{\Rn} \chi_{\widetilde \omega}(x) \d \MA(h_K[n-1],q;x)\\
&= n \int_{\interior(B^n)} \chi_{\omega}\left(\frac{x}{|x|}\right) \d\MA(h_K[n-1],q;x),
\end{align*}
where we have used Lemma~\ref{le:ma_vanish_o_s} in the second equality. We can now write
\begin{align*}
V(K[n-1],L)&=\frac{1}{n} \int_{\s^{n-1}} h_L(z) \d S_{n-1}(K,z)\\
&= \int_{\interior(B^n)} h_L\left(\frac{x}{|x|}\right) \d\MA(h_K[n-1],q;x)
\end{align*}
for $L\in\Kn$. The statement now follows by multilinearity.
\end{proof}

Another connection between mixed volumes and mixed Monge--Amp\`ere measures is due to the next result.
\begin{lemma}
\label{le:MA_bodies}
If $K_1,\ldots,K_n\in\Kn$, then
$$\MA(h_{K_1},\ldots,h_{K_n};B)= V(K_1,\ldots,K_n) \delta_o(B)$$
for Borel sets $B\subseteq \Rn$. In particular,
$$\binom{n}{j} \MA(h_K[j],h_{B^n}[n-j];B)=\kappa_{n-j} V_j(K) \delta_o(B)$$ 
for $0\leq j\leq n$.
\end{lemma}
\begin{proof}
For $K\in\Kn$ we have
\begin{equation}
\label{eq:cma_ind_K}
\MA(h_K;B)=\MAp(\ind_K;B)=\int_K \chi_B(\nabla \ind_K (x)) \d x = V_n(K) \delta_o(B)
\end{equation}
for Borel sets $B\subseteq \Rn$. Next, for $m\in\N$ let $\lambda_1,\ldots,\lambda_m\geq 0$ and $K_1,\ldots,K_m\in\Kn$.
We use \eqref{eq:def_mixed_ma} together with \eqref{eq:cma_ind_K} to obtain
\begin{align*}
\sum_{i_1,\ldots,i_n=1}^m \lambda_{i_1}\cdots \lambda_{i_n} \MA(h_{K_{i_1}},\ldots,h_{K_{i_n}};B) &= \MA(\lambda_1  h_{K_1}+ \cdots + \lambda_m h_{K_m};B)\\
&= \MA(h_{\lambda_1 K_1+\cdots + \lambda_m K_m};B)\\
&= V_n(\lambda_1 K_1 + \cdots + \lambda_m K_m)\delta_o(B)\\
&= \sum_{i_1,\ldots,i_n=1}^m \lambda_{i_1}\cdots \lambda_{i_n} V(K_{i_1},\ldots,K_{i_n}) \delta_o(B).
\end{align*}
The result now follows after comparing coefficients together with \eqref{eq:intrinsic_mixed}.
\end{proof}

For a similar result, where integrals with respect to mixed Monge--Amp\`ere measures are related to mixed volumes of convex bodies in $(n+1)$-dimensional space, we refer to \cite{Klartag_marginals}.

\subsection{Connections with Mixed Area Measures on \texorpdfstring{$\K^{n+1}$}{K(n+1)}}
\label{se:connections_ma2}
Let
$$\fconvcd = \{u\in\fconvs : \dom u \text{ is compact}\}.$$
We associate with each $u\in\fconvcd$ a convex body $K^u\in\K^{n+1}$. For this, we set $M_u=\max_{x\in\dom u} u(x)$, which is finite since $u$ has compact domain. We now define
\begin{equation}
\label{def:ku}
K^u=\epi u \cap\{x\in\R^{n+1} : \langle x,e_{n+1}\rangle \leq M_u\}.
\end{equation}
Next, let $u_1,u_2\in\fconvcd$. We have
$$\epi\big((\lambda_1 \sq u_1)\infconv (\lambda_2 \sq u_2)\big)=\lambda_1 \epi(u_1)+\lambda_2 \epi(u_2)$$
for $\lambda_1,\lambda_2\ge 0$ not both zero and, since the maximum of the sum of two functions is less than or equal to the sum of their maxima,
$$M_{(\lambda_1 \sq u_1)\infconv (\lambda_2 \sq u_2)}\leq \lambda_1 M_{u_1}+\lambda_2 M_{u_2}.$$
Thus, it follows from \eqref{def:ku} that for every $\lambda_1,\lambda_2\geq 0$ there exists a constant $c\geq 0$, depending on $u_1,u_2,\lambda_1,\lambda_2$, such that
\begin{equation}
\label{eq:sum_ku}
K^{\lambda_1 \sq u_1\infconv \lambda_2 \sq u_2} + [o,c\,e_{n+1}] =  \lambda_1 K^{u_1}+\lambda_2 K^{u_2}.
\end{equation}

\begin{remark}
Depending on the application, it might make sense to symmetrize $K^u$, i.e., to consider its union with an appropriate translate of its reflection at the hyperplane $e_{n+1}^\bot$. We remark that all the results presented here would remain unchanged if we considered this alternative definition for $K^u$.
\end{remark}

We need the following result, which was shown in an equivalent form in \cite[Corollary 3.7]{Knoerr_Ulivelli}. Here, $\s^n_{-}=\{\nu\in\s^n : \langle \nu,
e_{n+1}\rangle <0\}$ denotes the lower half-sphere of dimension $n$ and $\gnom\colon \s_-^n \to\Rn$,
$$
\gnom(\nu)= \frac{(\nu_1,\dots,\nu_n)}{|\nu_{n+1}|}
$$
denotes the gnomonic projection, where we write $\nu=(\nu_1,\ldots,\nu_{n+1})$ for $\nu\in\s_-^n$.

\begin{lemma}
\label{thm:ma_sn}
If $\varphi\colon \Rn \to [0,\infty)$ is measurable, then
$$\int_{\Rn}\varphi(y)\d\MAp(u;y)=\int_{\s_-^n} \tilde{\varphi}(\nu)\d S_n(K^u,\nu),$$
for $u\in\fconvcd$, where
\begin{equation}
\label{eq:tilde_zeta}
\tilde{\varphi}(\nu)=|\langle \nu,e_{n+1}\rangle|\varphi(\gnom(\nu))
\end{equation}
for $\nu\in\s_-^n$.
\end{lemma}
 
\begin{corollary}
\label{cor:mixed_ma_s}
If $\varphi\colon \Rn\to [0,\infty)$ is measurable, then
$$\int_{\Rn} \varphi(y) \d\MAp(u_1,\ldots,u_n;y)=\int_{\s_-^n} \tilde{\varphi}(\nu) \d S(K^{u_1},\ldots,K^{u_n},\nu)$$
for $u_1,\ldots,u_n\in \fconvcd$, where $\tilde{\varphi}\colon \s_-^n\to[0,\infty)$ is as in \eqref{eq:tilde_zeta}.
\end{corollary}
\begin{proof}
By \eqref{eq:cma_poly_exp} we have
\begin{multline*}
\int_{\Rn} \varphi(y) \d\MAp\big((\lambda_1 \sq u_1) \infconv \cdots \infconv (\lambda_m \sq u_m); y\big)\\
= \sum_{i_1,\ldots,i_n=1}^m \lambda_{i_1}\cdots \lambda_{i_n} \int_{\Rn} \varphi(y) \d\MAp(u_{i_1},\ldots,u_{i_n};y)
\end{multline*}
for $u_1,\ldots,u_m\in\fconvcd$, $\lambda_1,\ldots,\lambda_m\geq 0$, and $m\in\N$. On the other hand, we obtain from Lemma~\ref{thm:ma_sn}, \eqref{eq:sum_ku} and \eqref{eq:def_mixed_s} in dimension $n+1$, that
\begin{align*}
\int_{\Rn} \varphi(y) \d\MAp\big((\lambda_1 \sq u_1) \infconv \cdots &\infconv (\lambda_m \sq u_m); y\big)\\
&= \int_{\s_-^n} \tilde{\varphi}(\nu) \d S_n(K^{(\lambda_1 \sq u_1) \infconv \cdots \infconv (\lambda_m \sq u_m)}, \nu)\\
&= \int_{\s_-^n} \tilde{\varphi}(\nu) \d S_n(\lambda_1 K^{u_1} + \cdots + \lambda_m K^{u_m}, \nu)\\
&= \sum_{i_1,\ldots,i_n=1}^m \lambda_{i_1}\cdots \lambda_{i_n} \int_{\s_-^n} \tilde{\varphi}(\nu) \d S(K^{u_{i_1}},\ldots,K^{u_{i_n}},\nu).
\end{align*}
Here we have used that Lemma~\ref{le:mx_area_meas_proj} implies that for every $K\in \K^{n+1}$ and $c\geq 0$ the area measure $S_n(K,\cdot)$ coincides with $S_n(K+[o,c\,e_{n+1}],\cdot)$ on $\s_-^n$. The result now follows after comparing coefficients.
\end{proof}

We want to extend results, the proof of which is based on Corollary \ref{cor:mixed_ma_s}, from $\fconvcd$ to $\fconvs$. In order to do so, we will show that if we integrate a measurable $\varphi$ with compact support on $\Rn$ with respect to $\MAp(u_1,\ldots,u_n;\cdot)$ with $u_1,\ldots,u_n\in\fconvs$, then we may replace these functions by elements from $\fconvcd$ without changing the value of the integral. For this, we need to introduce a tool first. For $w\in\fconv$ and $r>0$, the \emph{Lipschitz regularization} $\reg_r w$ of $w$    is defined by
\begin{equation}
\label{eq:lip_reg}
\reg_{r} w(x) = (w^*+\ind_{1/r\,B^n})^*(x) =(w \infconv (h_{1/r \, B^n}))(x) 
\end{equation}
for $x\in\Rn$, which is an element of $\fconvf$. The name is justified since 
$$
\reg_{r} w(x)=
\sup \left\{a(x) : a \text{ is affine on } \Rn, a \leq w, |\nabla a|\leq \tfrac 1r \right\},$$
as pointed out in \cite[Section 4]{Colesanti-Ludwig-Mussnig-3}.  
It is a consequence of \cite[Proposition 4.1 (vi)]{Colesanti-Ludwig-Mussnig-3} that for given $m>0$ and $w\in\fconv$, such that $\dom(w)$ contains an open neighborhood of $m\,B^n$, we have
\begin{equation}
\label{eq:lip_reg_equiv}
\reg_{r} w \equiv w \quad \text{and} \quad \partial \reg_{r} w \equiv \partial w \quad \text{on } m\,B^n
\end{equation}
whenever $1/r$ is greater than the Lipschitz constant of $w$ on a neighborhood of $m\, B^n$. Note that every $w\in\fconv$ is locally Lipschitz on the interior of its domain (see, for example, \cite[Theorem 1.5.3]{Schneider_CB}).

\begin{lemma}
\label{le:fconvs_fconvcd}
For every $u_1,\dots,u_n \in \fconvs$ and measurable $\varphi\colon \Rn\to [0,\infty)$ with compact support, there exist $\bar{u}_1,\dots,\bar{u}_n \in \fconvcd$ such that
\begin{equation}
\label{eq:int_u_u_bar}
\int_{\Rn} \varphi(y) \d \MAp(u_1,\ldots,u_n;y)=\int_{\Rn} \varphi(y) \d \MAp(\bar{u}_1,\ldots,\bar{u}_n;y).
\end{equation}
In addition,
\begin{align*}
\int_E \varphi(y_E) & \d\MAp_E(\proj_E u_1,\ldots,\proj_E u_j;y_E)\\
&= \int_{E} \varphi(y_E) \d\MAp_E(\proj_E \bar{u}_1,\ldots,\proj_E \bar{u}_j;y_E)
\end{align*}
for $E\in\Grass{j}{n}$ with $1\leq j <n$.
\end{lemma}
\begin{proof}
Let $m>0$ be such that $\supp \varphi \subseteq m\,B^n$ and choose $r_i>0$ such that $1/r_i$ is greater than the Lipschitz constant of $u_i^*$ on an open neighborhood of $(m+1)\, B^n$ for $1\leq i\leq n$ (recall that $u_i^*\in\fconvf$). We define $\bar{u_i}\in\fconvcd$ as $\bar{u_i}=(\reg_{1/(m+2)} u_i)+\ind_{1/r_i\, B^n}$.  By \eqref{eq:sum_conj} and \eqref{eq:lip_reg} we have
$$\bar{u_i}^* = \reg_{r_i} (u_i^*+\ind_{(m+2)B^n})\in\fconvf$$
for $1\leq i\leq n$. By our choice of $r_i$ it now follows from \eqref{eq:lip_reg_equiv} that 
\begin{equation}
\label{eq:bar_u_reg_u_equal}
\bar{u}_i^*=\reg_{r_i} (u_i^* + \ind_{(m+2)\,B^n})  = u_i^*
\end{equation}
and
$$
\partial \bar{u}_i^*=\partial (u_i^* + \ind_{(m+2)\,B^n})  = \partial u_i^*
$$
on $(m+1)\, B^n$ and together with Lemma~\ref{le:sum_subdiff} we thus obtain
\begin{equation*}
\partial (\bar{u}_{i_1}^* + \cdots +\bar{u}_{i_k}^*) \equiv \partial (u_{i_1}^*+\cdots + u_{i_k}^*) \quad \text{on } (m+1)\, B^n
\end{equation*}
for $1\leq i_1 < \cdots < i_k \leq n$ and $1\leq k\leq n$.
Hence, by \eqref{eq:def_MA}, \eqref{eq:polarization_formula}, and \eqref{eq:def_mixed_map}, we have
$$
\MAp(u_1,\ldots,u_n;\cdot)=\MA(u_1^*,\ldots,u_n^*;\cdot)=\MA(\bar{u}_1^*,\ldots,\bar{u}_n^*;\cdot)=\MAp(\bar{u}_1,\ldots,\bar{u}_n;\cdot)
$$
on $m\,B^n$, which gives \eqref{eq:int_u_u_bar}.

For the statement regarding the projections, we fix a linear subspace $E\in\Grass{j}{n}$ and denote by $\partial_E$ the subdifferential with respect to the ambient space $E$. Equation \eqref{eq:bar_u_reg_u_equal} implies $\bar{u}_i^*\vert_E=u_i^*\vert_E$ on $(m+1)B^n\cap E$, and therefore $\partial_E(\bar{u}_i^*\vert_E)=\partial_E(u_i^*\vert_E)$ on (an open neighborhood of) $mB^n\cap E$, for $i=1,\ldots,n$. Hence, an application of Lemma~\ref{le:sum_subdiff} in $E$ shows that 
$$
\partial_E (\bar{u}_{i_1}^*\vert_E + \cdots +\bar{u}_{i_k}^*\vert_E) \equiv \partial_E (u_{i_1}^*\vert_E+\cdots + u_{i_k}^*\vert_E) \quad \text{on } m\, B^n \cap E,
$$
for $1\leq i_1 < \cdots < i_k \leq n$ and $1\leq k\leq j$. Together with Lemma~\ref{le:proj_conj_restr} the proof is now analogous to the considerations above.
\end{proof}

\subsection{Behavior under Orthogonal Projections}
\label{se:behavior_proj}
We provide results for conjugate mixed Monge--Amp\`ere measures for orthogonal projections of super-coercive convex functions. Equivalently (see Lemma \ref{le:proj_conj_restr}), this means that we consider mixed Monge--Amp\`ere measures of restrictions of finite-valued convex functions to linear subspaces.

The results of this section can be obtained from Lemma~\ref{le:mx_area_meas_proj} together with Corollary~\ref{cor:mixed_ma_s}. However, we think that the proofs below are of independent interest. In particular, this allows for direct proofs of Theorem~\ref{thm:ck_ma} and Theorem~\ref{thm:ma_supp} that do not rely on results for area measures of convex bodies (see Remark~\ref{re:direct_proof}).

In the proof of Lemma \ref{le:mx_ma_meas_restr} below, the next auxiliary result is used. It is a straightforward consequence of \cite[Theorem 3.3]{EG2015} and Rademacher's theorem.

\begin{lemma}\label{lem:EG15}
If $U\subseteq \R^n$ is open and $g,h\colon U\to\R$ are locally Lipschitz, then, for almost all $x\in \{z\in U: g(z)=h(z)\}$, $g$ and $h$ are differentiable at $x$ and $\nabla g(x)=\nabla h(x)$.    
\end{lemma}
 
\begin{lemma}
\label{le:mx_ma_meas_restr}
If $e\in\sn$ and $E=e^\perp$, then
$$n \MAp(u[n-1],\ind_{[o,e]};B)=\MAp_E(\proj_{E} u;B\cap E)$$
for $u\in\fconvs$ and Borel sets $B\subseteq \Rn$.
\end{lemma}

\begin{proof}
Throughout the following, we fix an arbitrary Borel set $B\subseteq \Rn$ and $u\in\fconvs$. Without loss of generality we set $e=e_n$ and thus $\operatorname{span}\{e_1,\ldots,e_{n-1}\}=E$. We set $D_u=\proj_E \dom(u)= \dom (\proj_E u)$. 

For $x_E\in D_u$ the set $\argmin (r\mapsto u(x_E+re))$ is a compact interval in $\R$. We define $m_u(x_E)$ as its midpoint, which gives a Borel measurable map $m_u\colon D_u\to \R$ (we postpone the proof of measurability to Lemma~\ref{Le:mb}). Let $\varepsilon>0$. For $x_E\in D_u$ and $r\in\R$ we obtain
\begin{align*}(u\infconv \ind_{[o,\varepsilon e]})(x_E +re)&=\inf\{u(x_E+re-z)+\ind_{[o,\varepsilon e]}(z) : z\in \Rn \}\\
&=\inf\{u(x_E+(r-s)e) : s\in [o,\varepsilon]\}\\
&=\begin{cases} u(x_E+re),\quad &\text{if } r \leq m_u(x_E),\\
\proj_E u(x_E),\quad &\text{if } m_u(x_E) < r < m_u(x_E) +\varepsilon,\\
u(x_E+(r-\varepsilon)e),\quad &\text{if } m_u(x_E)+\varepsilon \leq r,
\end{cases}\end{align*}
where $\proj_E u(x_E)=\min_{y\in \lin\{e\}} u(x_E+y)$. 
Observe that $\proj_E \dom(u\infconv \ind_{[o,\varepsilon e]})=D_u$. Hence, by Fubini's theorem,  
\begin{align}
\label{eq:ma_infconv_fubini}
&\MAp(u\infconv \ind_{[o,\varepsilon e]};B)\nonumber\\
&= \int_{\dom (u\infconv \ind_{[o,\varepsilon e]})} \chi_B(\nabla (u\infconv \ind_{[o,\varepsilon e]})(x)) \d x\nonumber\\
&= \int_{D_u} \int_{\dom (u\infconv \ind_{[o,\varepsilon e]}) \cap (x_E + \lin\{e\})}\chi_B(\nabla (u\infconv \ind_{[o,\varepsilon e]})(x_E+y)) \d y \d x_E.
\end{align}
Let $U=\interior\dom (u\infconv \ind_{[o,\varepsilon e]})\cap \interior\dom (u)$. The functions $u\infconv \ind_{[o,\varepsilon e]}$ and $u$ are locally Lipschitz on $U$. For $A=\{x_E+re\in U:r\le m_u(x_E)\}$, we obtain 
$$A\subseteq \{x_E+re\in U:
(u\infconv \ind_{[o,\varepsilon e]})(x_E+re)=u(x_E+re)\}.$$
Lemma \ref{lem:EG15} yields that 
$\nabla (u\infconv \ind_{[o,\varepsilon e]})(x_E +re) = \nabla u(x_E+re)$ 
for almost all $x_E+re\in A$. Since the boundary of $\dom(u)$ has Lebesgue measure zero and 
$$\{x_E+re\in   \interior\dom (u\infconv \ind_{[o,\varepsilon e]})\setminus \dom(u):r\le m_e(x_E)\}=\emptyset,$$
it follows that 
$$\nabla (u\infconv \ind_{[o,\varepsilon e]})(x_E +re) = \nabla u(x_E+re)$$
for almost all $x_E+re\in  \interior\dom (u\infconv \ind_{[o,\varepsilon e]})$ such that $r\le m_e(x_E)$. 

In the same way, we obtain that 
$$\nabla (u\infconv \ind_{[o,\varepsilon e]})(x_E +re) = \nabla u(x_E+(r-\varepsilon)e)$$
holds for almost all $x_E+re\in  \interior\dom (u\infconv \ind_{[o,\varepsilon e]})$ such that $m_u(x_E)+\varepsilon \leq r$.

Moreover, on $(m_u(x_E),m_u(x_E) +\varepsilon)$, the function
$$r\mapsto (u\infconv \ind_{[o,\varepsilon e]})(x_E +re)$$
is constant. Therefore, applying once again Lemma \ref{lem:EG15},
$$\nabla (u\infconv \ind_{[o,\varepsilon e]})(x_E +re) = \begin{pmatrix}\nabla_{E} \proj_E u(x_E) \\ 0\end{pmatrix}$$
holds for almost all $x_E+re\in  \interior\dom (u\infconv \ind_{[o,\varepsilon e]})$ such that $r\in (m_u(x_E),m_u(x_E)+\varepsilon)$.

As a consequence, we obtain
\begin{align}
\label{eq:int_dom_split}
&\int_{\dom (u\infconv \ind_{[o,\varepsilon e]}) \cap (x_E + \lin\{e\})}\chi_B(\nabla (u\infconv \ind_{[o,\varepsilon e]})(x_E+y)) \d y \nonumber\\
&= \int_{\dom(u) \cap (x_E + \lin\{e\})} \chi_B(\nabla u(x_E+y)) \d y  + \varepsilon \chi_B\left(\begin{pmatrix}\nabla_{E} \proj_E u(x_E) \\ 0\end{pmatrix}\right)
\end{align}
for $x_E\in E$. We rewrite
$$ 
\chi_B\left(\begin{pmatrix}\nabla_{x_E}\proj_E u(x_E)\\ 0\end{pmatrix}\right) = \chi_{B\cap E}(\nabla_E \proj_E u(x_E)).
$$ 
Therefore it follows from \eqref{eq:ma_infconv_fubini} and \eqref{eq:int_dom_split} that
\begin{align*}
&\MAp(u\infconv \ind_{[o,\varepsilon e]};B)\\
&= \int_{D_u} \Big(\int_{\dom(u) \cap (x_E + \lin\{e\})} \chi_B(\nabla u(x_E+y))\d y
 + \varepsilon \chi_{B\cap E}(\nabla_E \proj_E u(x_E))\Big) \d x_E\\
&=\int_{\dom u} \chi_B(\nabla u(x)) \d x + \varepsilon \int_{\dom(\proj_E u)} \chi_{B\cap E}(\nabla_E \proj_E u(x_E))\d x_E\\
&=\MAp(u;B) + \varepsilon\MAp_E(\proj_E u;B\cap E).
\end{align*}
Hence it follows from \eqref{eq:cma_poly_exp} that
\begin{align*}
n \MAp(u[n-1],\ind_{[o,e]};B)&= \lim\nolimits_{\varepsilon\to 0^+}\frac{\MAp(u\infconv \ind_{[o,\varepsilon e]};B)-\MAp(u;B)}{\varepsilon}\\
&=\MAp_E(\proj_E u;B\cap E),
\end{align*}
which concludes the proof.
\end{proof}

\begin{lemma}\label{Le:mb}
    The map $m_u\colon D_u\to\R$ is measurable.
\end{lemma}
\begin{proof}
Let $\cFn$ denote the space of closed subsets of $\Rn$ together with the Fell topology (see \cite[Section 12.2]{schneider_weil}), measurability refers to the induced Borel $\sigma$-algebra. The map $\cEu\colon D_u\to\cFn$,   defined by 
$$\cEu(x_E)=\epi(u)\cap (x_E+\lin\{e,e_{n+1}\})$$ 
for $x_E\in D_u$, is upper semicontinuous (see \cite[Theorem 12.2.6]{schneider_weil}) and hence measurable. Moreover, the map $h\colon \cFn\to[-\infty,\infty]$, defined by $$h(F)=\inf\{\langle z,e_{n+1}\rangle:z\in F\}$$ for $F\in\cFn$, is upper semicontinuous (and hence measurable). To verify this, we show that 
$\limsup_{i\to\infty}h(F_i)\le h(F)$ whenever $F_i\to F$ in $\cFn$ as $i\to\infty$. If $F= \emptyset$, then $h(F)=\infty$, and there is nothing to show. Now let $F\neq\emptyset$ and $h(F)<\alpha$ for some $\alpha\in\R$. Hence  there is some $z\in F$ such that $\langle z,e_{n+1}\rangle<\alpha$. Since $F_i\to F$, there are $z_i\in F_i$ (if $i$ is large enough) such that $z_i\to z$, hence $\langle z_i,e_{n+1}\rangle<\alpha$ (if $i$ is large enough). This shows that if $i$ is large enough, then $h(F_i)<\alpha$, which proves the claim. 

For $x_E\in D_u$ we define 
$$
\cMu(x_E)=\left(\cEu-h(\cEu(x_E))e_{n+1}\right)\cap \lin\{e\}\in\Kn.
$$
It follows from \cite[Theorem 12.2.6 and Theorem 12.3.1]{schneider_weil}  in combination with the fact that $\Kn$ is a measurable subset of $\cFn$ (see \cite[Theorem 2.4.2]{schneider_weil}), that $\cMu\colon D_u\to\Kn$ is measurable. Writing $c\colon \Kn\to\R^n$ for the (continuous) map which assigns to a convex body the center of its circumscribed ball (see \cite[Lemma 4.1.1]{schneider_weil}), we obtain $m_u=c\circ \cMu$, and hence $m_u$ is measurable as a composition of measurable maps.
\end{proof}

\begin{remark}
For an alternative proof of Lemma~\ref{Le:mb}, we observe first that it follows from standard properties of convex conjugates (see, for example, \cite[Section 11]{RockafellarWets}) that
$$\argmin\nolimits_{r\in E^\perp} u(x_E+r)=\partial_{E^\perp} \big(\proj_{E^\perp} (u^*(\cdot) - \langle \cdot,x_E\rangle \big)(o)$$
for $x_E\in\proj_E \dom(u)$. One can now show that composing the above with the map that assigns to a convex body the center of its circumscribed ball gives a measurable map.
\end{remark}

From Lemma \ref{Le:mb} we deduce a more general version for conjugate mixed Monge--Amp\`ere measures.

\begin{theorem}
\label{theo:mixed_proj_inter}
If $e\in\sn$ and $E=e^\perp$, then
$$n\MAp(u_1,\ldots,u_{n-1},\ind_{[o,e]};B)=\MAp_E(\proj_{E} u_1,\ldots,\proj_{E} u_{n-1};B\cap E)$$
for $u_1,\ldots,u_{n-1}\in \fconvs$ and Borel sets $B\subseteq \Rn$.
\end{theorem}

\begin{proof}
Let $\lambda_1,\ldots,\lambda_{n-1}\ge 0$, and let $u_1,\ldots,u_{n-1}\in \fconvs$ and Borel sets $B\subseteq \Rn$ be given. An application of Lemma \ref{Le:mb} and Lemma \ref{le:proj_conv} yield
\begin{align*}
&n\MAp\left((\lambda_1\sq u_1)\infconv\cdots \infconv (\lambda_{n-1}\sq u_{n-1}),\ind_{[o,e]};B\right)\\
&=\MAp_E\left((\lambda_1\sq \proj_E u_1)\infconv\cdots \infconv (\lambda_{n-1}\sq \proj_E u_{n-1});B\cap E\right).
\end{align*}
Multinomial expansion and comparison of coefficients imply the assertion.
\end{proof}

Iterated application of Theorem \ref{theo:mixed_proj_inter} yields the following result.

\begin{corollary}\label{cor:4.15}
 Let $1\le j\le n-1$. If $z_{j+1},\ldots,z_n\in\sn$ are pairwise orthogonal and $E_j=\lin\{z_{j+1},\ldots,z_n\}^\perp$, then 
$$n!\MAp(u_1,\ldots,u_{j},\ind_{[o,z_{j+1}]},\ldots,\ind_{[o,z_{n}]};B)=j!\MAp_{E_j}(\proj_{E_j} u_1,\ldots,\proj_{E_j} u_{j};B\cap E_j)$$
for $u_1,\ldots,u_{j}\in \fconvs$ and Borel sets $B\subseteq \Rn$. In particular, the measure on the left side is independent of the specific choice of an orthonormal basis $z_{j+1},\ldots,z_n$ of $E_j^\perp$. 
\end{corollary}

We also state explicitly the corresponding result for the mixed Monge--Amp\`ere measures.

\begin{corollary}\label{cor:4.15dual}
 Let $1\le j\le n-1$. If $z_{j+1},\ldots,z_n\in\sn$ are pairwise orthogonal and $E_j=\lin\{z_{j+1},\ldots,z_n\}^\perp$, then 
$$n!\MA(v_1,\ldots,v_{j},h_{[o,z_{j+1}]},\ldots,h_{[o,z_{n}]};B)=j!\MA_{E_j}(v_1\vert_{E_j},\ldots, v_{j}\vert_{E_j};B\cap E_j)$$
for $v_1,\ldots,v_{j}\in \fconvf$ and Borel sets $B\subseteq \Rn$. In particular, the measure on the left side is independent of the specific choice of an orthonormal basis $z_{j+1},\ldots,z_n$ of $E_j^\perp$. 
\end{corollary}

\section{Main Results for Mixed Monge--Amp\`ere Measures}
\label{se:main_res_ma}

\subsection{Kubota-type Formulas}
\label{se:kubota_ma}
The aim of this section is to prove Theorem~\ref{thm:ck_ma}. By Lemma~\ref{le:proj_conj_restr}, we may consider the equivalent dual statement for conjugate mixed Monge--Amp\`ere measures instead.

\begin{theorem}
\label{thm:ck_map}
If $1\leq k <n$ and $\varphi\colon \Rn\to [0,\infty)$ is measurable, then
\begin{multline*}
\frac{1}{\kappa_n} \int_{\Rn} \varphi(y) \d\MAp(u_1,\ldots,u_k,\ind_{B^n}[n-k];y)\\
=\frac{1}{\kappa_k} \int_{\Grass{k}{n}}\int_E \varphi(y_E) \d\MAp_E(\proj_E u_1,\ldots,\proj_E u_k;y_E) \d E
\end{multline*}
for $u_1,\ldots,u_k\in\fconvs$.
\end{theorem}
\begin{proof}
Throughout the proof, let $1\leq k<n$ and let $\varphi\colon\Rn\to[0,\infty)$ be measurable. By the monotone convergence theorem, we may assume without loss of generality that $\varphi$ has compact support. It follows from Corollary~\ref{cor:mixed_ma_s} and Lemma~\ref{le:fconvs_fconvcd} that for $u_1, \ldots ,u_k \in\fconvs$ there exist $K^{\bar{u}_1},\ldots, K^{\bar{u}_k}\in\K^{n+1}$ (possibly depending on $\varphi$) such that
$$\int_{\Rn} \varphi(y) \d\MAp(u_1,\ldots,u_k,\ind_{B^n}[n-k];y) = \int_{\s_{-}^n} \tilde{\varphi}(\nu) \d S(K^{\bar{u}_1},\ldots,K^{\bar{u}_k},B_{\Rn}^n[n-k],\nu),$$
where $\tilde{\varphi}(\nu)=|\langle \nu,e_{n+1}\rangle|\varphi(\gnom(\nu))$ for $\nu\in \s_{-}^n$ and where $B_{\Rn}^n$ denotes the $n$-di\-men\-sion\-al unit ball in $e_{n+1}^\perp$.

Next, it easily follows from \eqref{eq:proj_epi} and \eqref{def:ku} that there are constants $c_i\ge 0$ such that 
$$K^{\proj_E \bar{u}_i}+[o,c_ie_{n+1}]=\proj_{E\times\lin\{e_{n+1}\}} K^{\bar{u}_i}$$
for every $1\leq i\leq k$ and $E\in\Grass{k}{n}$. Again from Corollary~\ref{cor:mixed_ma_s}, this time applied with respect to the ambient space $E$, and Lemma~\ref{le:fconvs_fconvcd} we obtain
\begin{align*}
\int_E &\varphi(y_E)\d\MAp_E(\proj_E u_1,\ldots,\proj_E u_k;y_E)\\
&= \int_{(\s_E^k)_-} \tilde{\varphi}(\nu_E)\d S_{E\times\lin\{e_{n+1}\}}(\proj_{E\times\lin\{e_{n+1}\}} K^{\bar{u}_1},\ldots,\proj_{E\times\lin\{e_{n+1}\}} K^{\bar{u}_k},\nu_E),
\end{align*}
where also Lemma~\ref{le:mx_area_meas_proj} was used. 
The result now follows from Theorem~\ref{thm:ck_sam}, applied with respect to the ambient space $\R^{n+1}$.
\end{proof}

\begin{remark}
    For an alternative proof of Theorem \ref{thm:ck_map} we remark first that it is enough to show the result for $\varphi\in C_c(\Rn)$. Next, we can choose sequences of functions $u_{j,i}\in \fconvcd$, $i\in\N$, such that $u_{j,i}$ epi-converges to $u_j$ as $i\to\infty$. By \cite[Lemma 3.2]{Colesanti-Ludwig-Mussnig-6}, then also $\proj_E u_{j,i}$ epi-converges to $\proj_E u_j$ with respect to $E$. Applying  \cite[Theorem 5.2 (d)]{Colesanti-Ludwig-Mussnig-7} on the left-hand side, using  \cite[Theorem 5.2 (d)]{Colesanti-Ludwig-Mussnig-7} on the right-hand side for each fixed $E$, as well as the dominated convergence theorem (on the right-hand side), we obtain the asserted relation in full generality.
    
 The application of the dominated convergence theorem can be justified as follows. Let $A\subset\R^n$ be compact and let $u_{j,i}\to u_j$ as $i\to\infty$ with respect to epi-convergence in $\fconvs$, for $j=1,\ldots,k$. It follows that $u_{j,i}^*$ epi-converges to $u_j^*$ as $i\to\infty$ in $\fconvf$ and therefore 
 $$
 \partial u_{j,i}^*(A)\subseteq \partial u_j^*(A)+B^n\subseteq c\cdot B^n,
 $$
if $i$ is sufficiently large, with a constant $c$ depending on $u_j$ and $A$ (the first inclusion can be deduced by a compactness argument from \cite[Theorem 24.5]{Rockafellar}). Using Lemma \ref{le:proj_conj_restr} and Lemma \ref{lem:projpartial} below, we then obtain that
$$
\partial (\proj_E u_{j,i})^*(x_E)=\partial (u_{j,i}^*)\vert_E(x_E)=
\proj_E(\partial u_{j,i}^*(x_E))\subseteq \proj_E(c B^n)=c B^n\cap E$$
 for all $x_E\in A\cap E$, if $i$ is sufficiently large. Hence
 $$
 \MAp_E(\proj_E u_{1,i},\ldots,\proj_E u_{k,i};A)\le c(u_1,\ldots,u_k,A)<\infty,
 $$
if $i$ is large enough, with a constant $c(u_1,\ldots,u_k,A)$ that is independent of $E$.
\end{remark}

 \begin{lemma}\label{lem:projpartial}
    Let $1\leq k\leq n-1$ and $E\in\Grass{k}{n}$. If $v \in \fconvf$ and $x \in E$, then
    \begin{equation}\label{eq:section_projection}
        \partial v\vert_E (x) =\proj_E \partial v(x).
    \end{equation}
\end{lemma}
\begin{proof}
    If $p \in \proj_E \partial v(x)$, then there exists $\bar{p} \in \partial v(x)$ such that $(\bar{p}-p)\perp E$ and
    $$v(y) \geq v(x) + \langle\bar{p},y-x \rangle$$
    for $y\in \R^n$. In particular, if $y \in E$, then
    $$v\vert_E(y)= v(y)\geq v(x) + \langle\bar{p},y-x \rangle=v\vert_E(x)+\langle p,y -x\rangle.$$
    Therefore, $p \in \partial v\vert_E (x)$.
    
    Conversely, if $p \in \partial v\vert_E (x)$, then $\langle p,\cdot \rangle$ is a linear functional on $E$ such that
    $$v(y)-v(x)\geq \langle p,y-x \rangle$$
    for every $y \in E$. By the Hahn--Banach theorem \cite[Section 12.31]{Schechter} (equivalently, by a separation argument), there exists $\bar{p} \in \R^n$ such that
    $$v(y)-v(x)\geq \langle \bar{p},y-x \rangle$$
    for every $y \in \R^n$ and $\langle \bar{p},y-x\rangle=\langle p, y-x\rangle$ for every $y \in E$. In particular, $\bar{p} \in \partial v(x)$ and $p =\proj_E \bar{p} \in \proj_E \partial v(x)$, proving the claim.
\end{proof}

\begin{remark}
We obtain further formulas by combining Corollary \ref{cor:4.15} and Theorem \ref{thm:ck_map} and replacing the integration over the Grassmannian with an integration over (partial) orthonormal frames (a Stiefel manifold). For $k\in \{1,n-1\}$, the integration over the corresponding Stiefel manifolds can be expressed in terms of an integral over the unit sphere. Specifically, for $k=1$ we have
\begin{align*}   
&\int_{\Rn} \varphi(y) \d\MAp(u,\ind_{B^n}[n-1];y)\\
&\qquad =\frac{(n-1)!}{2} \int_{\sn}\int_{\R^n} \varphi(y) \d\MAp(u,\ind_{[o,z_1(e)]},\ldots,\ind_{[o,z_{n-1}(e)]};y) \d \hm^{n-1}(e),
\end{align*}
where $z_1(e),\ldots,z_{n-1}(e)$ is an arbitrary orthonormal basis of $e^\perp$ (which can but need not be chosen as a measurable function of $e\in \sn$). 

The case $k=n-1$ can be stated in the form
\begin{align*}   
&\int_{\Rn} \varphi(y) \d\MAp(u_1,\ldots,u_{n-1},\ind_{B^n};y)\\
&\qquad =\frac{1}{\kappa_{n-1}}\int_{\sn}\int_{\R^n} \varphi(y) \d\MAp(u_1,\ldots,u_{n-1},\ind_{[o,e]};y) \d \hm^{n-1}(e).
\end{align*}
\end{remark}
 
\subsection{Comparison with Crofton-type Formulas}
\label{se:compare_crofton}
The following was shown in \cite[Theorem 2.1]{Colesanti-Hug_MM}. Note that in \cite{Colesanti-Hug_MM} more general semi-convex functions were treated (in which case signed measures are obtained) and thus the Borel set $\eta\subset\Rn$ is assumed to be relatively compact. For convex functions, such an assumption is not necessary. Recall that $q(x)=|x|^2/2$ for $x\in\Rn$.

\begin{theorem}
If $1\leq k\leq n$, $0\leq j \leq k$, and $\eta\subset\Rn$ is Borel measurable, then
$$\MA(v[k-j], q[n-k+j];\eta)=\int_{\Aff{k}{n}} \MA_E(v\vert_E[k-j],q\vert_E[j];\eta\cap E)\d\mu_k(E)$$
for $v\in\fconvf$.
\end{theorem}

We restrict to the case $j=0$ above. It is easy to see that one can obtain the following generalization.

\begin{theorem}
If $1\leq k <n$ and $\eta\subset \Rn$ is Borel measurable, then
$$\MA(v_1,\ldots,v_k,q[n-k];\eta)=\int_{\Aff{k}{n}} \MA_E(v_1\vert_E,\ldots,v_k\vert_E;\eta\cap E)\d\mu_k(E)$$
for $v_1,\ldots,v_k\in\fconvf$.
\end{theorem}
This result should be compared with Theorem~\ref{thm:ck_ma}. In particular, it becomes evident that - up to normalization - if $(n-k)$ entries of $h_{B^n}$ in a mixed Monge--Amp\`ere measure are replaced with $q$, then in the integral geometric formula the Grassmannian $\Grass{k}{n}$ is replaced with the affine Grassmannian $\Aff{k}{n}$.

\subsection{Supports of Mixed Monge--Amp\`ere Measures}
\label{se:supp_ma}
We start this section with a simple consequence of Corollary~\ref{cor:mixed_ma_s}. Throughout the following, we use the convention that for a measure space $(X,\mathcal{M},\mu)$ and $A\in\mathcal{M}$, we write $\supp \mu\vert_A$ for the support (with respect to the ambient space $A$) of the restriction of $\mu$ to $A$ (considered as a measure on the $\sigma$-algebra induced on $A$).

\begin{corollary}
\label{cor:supp_cma_gnom}
For every $u_1,\ldots,u_n\in\fconvcd$,
$$\supp \MAp(u_1,\ldots,u_n;\cdot) = \gnom\left(\supp S(K^{u_1},\ldots,K^{u_n},\cdot)\big\vert_{\s_-^n}\right).$$
\end{corollary}

For $K\in\K^{n+1}$ and $j\in\{0,\ldots,n\}$, let $\Sdisk_j(K,\cdot) = S(K[j],B_{\Rn}^n[n-j],\cdot)$, where $B_{\Rn}^n$ denotes the $n$-di\-men\-sion\-al unit ball in $e_{n+1}^\perp$.
\begin{lemma}
\label{le:supp_map_j}
Let $u\in\fconvs$. For every bounded and open set $A\subset\Rn$ there exists a body $K_A^u\in\K^{n+1}$ such that
$$\supp \MAp(u[j],\ind_{B^n}[n-j];\cdot)\vert_A = \gnom\left( \supp \Sdisk_j(K_A^u,\cdot)\vert_{\gnom^{-1}(A)} \right)$$
for $1\leq j\leq n$. In addition, we may choose $K_A^{\proj_E u}=\proj_{E\times\lin\{e_{n+1}\}} K_A^u$
and
\begin{multline*}
\supp \MAp_E(\proj_E u;\cdot)\vert_{A\cap E}\\
= \gnom\left( \supp S_{E\times\lin\{e_{n+1}\}}(\proj_{E\times\lin\{e_{n+1}\}} K_A^u,\cdot)\vert_{\gnom^{-1}(A)\cap (E\times\lin\{e_{n+1}\})} \right)
\end{multline*}
for $E\in\Grass{j}{n}$.
\end{lemma}
\begin{proof}
It follows from Lemma~\ref{le:fconvs_fconvcd} that for every bounded $A\subset\Rn$ there exists a function $\bar{u}_A\in\fconvcd$ such that
$$\MAp(u[j],\ind_{B^n}[n-j];\cdot)\vert_A=\MAp(\bar{u}_A[j],\ind_{B^n}[n-j];\cdot)\vert_A$$
for every $1\leq j\leq n$. The first part of the statement now follows from Corollary~\ref{cor:supp_cma_gnom}. The statement regarding projections is a consequence of the second part of Lemma~\ref{le:fconvs_fconvcd}, \eqref{eq:proj_epi}, \eqref{def:ku}, and the second part of Lemma~\ref{le:mx_area_meas_proj}.
\end{proof}

We will use the following definition due to \cite{Colesanti-Hug_JLM}. For $v\in\fconvf$ we say that $x\in\Rn$ is a \emph{$j$-extreme point} of $v$ if there is no $(j+1)$-dimensional ball centered at $x$, such that the restriction of $v$ to this ball is an affine function. In addition, a $0$-extreme point is simply called \emph{extreme}. The set of all extreme points of $v$ is denoted by $\ext(v)$. The following functional analog of Theorem~\ref{thm:supp_s_j} was established in \cite[Theorem 2]{Colesanti-Hug_JLM}.

\begin{theorem}
\label{thm:supp_hess_meas}
Let $v\in\fconvf$ and $1\leq j\leq n$. The support of
$\MA(v[j],q[n-j];\cdot)$ is the closure of the set of all $(n-j)$-extreme points of $v$.
\end{theorem}

As a first main result of this section, we prove the following description of the support of the measures $\MA(v[j],h_{B^n}[n-j];\cdot)$ with $v\in\fconvf$ and $1\leq j\leq n$.

\begin{theorem}
\label{thm:supp_MAj_rest}
If $1\leq j\leq n$ and $v\in\fconvf$, then 
\begin{align*}
&\supp \MA(v[j],h_{B^n}[n-j];\cdot)\\
&\qquad=\cl\left(\{z\in \R^n:\exists E\in\Grass{j}{n} \text{ \rm such that }z\in\ext_E(\nu\vert_E)\}\right).
\end{align*}
\end{theorem}
\begin{proof}
Let $z_0\in\R^n$ and let $A\subset \Rn$ be bounded and open such that $A$ contains $z_0$ in its interior. It follows from definition \eqref{eq:def_mixed_map} and Lemma~\ref{le:supp_map_j} that
\begin{align*}
\supp \MA(v[j],h_{B^n}[n-j];\cdot)\vert_A&=\supp \MAp(v^*[j],\ind_{B^n}[n-j];\cdot)\vert_A\\
&= \gnom\left(\supp \Sdisk_j(K_A^{v^*},\cdot)\vert_{\gnom^{-1}(A)}\right),
\end{align*}
which means that $z_0\in\supp \MA(v[j],h_{B^n}[n-j];\cdot)$ if and only if
$$\gnom^{-1}(z_0)\in \supp \Sdisk_j(K_A^{v^*},\cdot).$$
By Theorem~\ref{thm:descrp_supp_s_proj} and Remark~\ref{re:descrp_supp_s_proj}, this is satisfied if and only if $z_0$ is in the closure of the set of all $z\in\R^n$ for which there exists a $E\in\Grass{j}{n}$ such that
$$\gnom^{-1}(z)\in\supp S_{E\times \lin\{e_{n+1}\}}(\proj_{E\times \lin\{e_{n+1}\}} K_A^{v^*}[j],\cdot).$$
By Lemma~\ref{le:proj_conj_restr} and Lemma~\ref{le:supp_map_j}, this is the case if and only if
$$z\in\supp\MA_E(v\vert_E).$$
The result now follows from Theorem~\ref{thm:supp_hess_meas}.
\end{proof}

Note that Theorem~\ref{thm:ma_supp} can be obtained as an immediate consequence of Theorem~\ref{thm:supp_MAj_rest}. Next, we state and prove an equivalent version which concerns the measures $\MAp(u[j],\ind_{B^n}[n-j];\cdot)$ with $u\in\fconvs$.

\begin{theorem}\label{thm:finmonotone}
If $1\leq j\leq k\leq n$, then
$$\supp \MAp(u[k],\ind_{B^n}[n-k];\cdot)\subseteq \supp \MAp(u[j],\ind_{B^n}[n-j];\cdot)$$
for $u\in\fconvs$.
\end{theorem}
\begin{proof}
The statement is an immediate consequence of Lemma~\ref{le:supp_map_j} and Corollary~\ref{cor:supp_s_nested}.
\end{proof}

\begin{remark}
\label{re:direct_proof}
The proof of Theorem~\ref{thm:ma_supp} - taking the proof of Corollary~\ref{cor:supp_s_nested} into account - almost exclusively concerns convex bodies. We only translate our results into the functional world in the very last step. Let us point out that Theorem~\ref{thm:ma_supp} can also be proven directly in the analytic setting by essentially replicating the proof of Corollary~\ref{cor:supp_s_nested} in the realm of functions. First, the necessary Kubota-type formula, Theorem~\ref{thm:ck_ma}, needs to be established. For this one proceeds as in Section~\ref{se:kubota}, but instead of  Lemma~\ref{le:mx_area_meas_proj} one uses Lemma~\ref{le:mx_ma_meas_restr} and the somewhat unwieldy Grassmannian $\Gr(\ell,j)$ is replaced with the more natural $\Grass{j}{n}$. Next, one uses ideas similar to those presented in Section~\ref{se:desc_supp}: with the help of Theorem~\ref{thm:supp_hess_meas} one shows that if an open set $\omega\subseteq\Rn$ and a function $v\in\fconvf$ are such that $\MA_E(v\vert_E;\omega\cap E)=0$ for every $E\in\Grass{n-1}{n}$, then also $\MA(v;\omega)=0$. Together with Theorem~\ref{thm:ck_ma}, this then implies Theorem~\ref{thm:ma_supp}.
\end{remark}

\subsection*{Acknowledgments}
The authors want to thank Rolf Schneider for valuable discussions. Furthermore, they are grateful to Jonas Knoerr, Dylan Langharst, and Ramon van Handel for their helpful comments. Parts of this project were carried out while the authors visited the Institute for Computational
and Experimental Research in Mathematics in Providence, RI, during the Harmonic Analysis and Convexity program in Fall 2022. Daniel Hug was supported by DFG research grant HU 1874/5-1
(SPP 2265). Fabian Mussnig was supported by the Austrian Science Fund (FWF): 10.55776/J4490 and 10.55776/P36210. Jacopo Ulivelli was supported, in part, by the Austrian Science Fund (FWF): 10.55776/P34446.


\bigskip\bigskip
\parindent 0pt\footnotesize

\parbox[t]{10cm}{
Daniel Hug\\
Institut f\"ur Stochastik, Karlsruhe Institute of Technology (KIT)\\ Englerstra{\ss}e 2, 76128 Karlsruhe, Germany\\
e-mail: daniel.hug@kit.edu}

\bigskip

\parbox[t]{10cm}{
Fabian Mussnig\\
Institut f\"ur Diskrete Mathematik und Geometrie, TU Wien\\
Wiedner Hauptstra{\ss}e 8-10, E104-06, 1040 Wien, Austria\\
e-mail: fabian.mussnig@tuwien.ac.at}

\bigskip

\parbox[t]{10cm}{
Jacopo Ulivelli\\
Institut f\"ur Diskrete Mathematik und Geometrie, TU Wien\\
Wiedner Hauptstra{\ss}e 8-10, E104-06, 1040 Wien, Austria\\
e-mail: jacopo.ulivelli@tuwien.ac.at}

\end{document}